\documentclass[10pt,leqno,a4paper]{amsart}
\usepackage[OT2,T1]{fontenc}
\usepackage{amsfonts, amssymb, amsmath, amsthm, latexsym, enumerate, epsfig, color, mathrsfs, fancyhdr, pdfsync, eurosym, endnotes, stmaryrd, bm, bbm}
\usepackage[all]{xy}

\usepackage{graphicx}
\usepackage{epstopdf}
\usepackage{setspace}
\usepackage{hyperref}
\usepackage{pdfsync}
\definecolor{labelkey}{rgb}{0,1,0}
\usepackage{caption}
\usepackage{amsmath,amscd,amssymb}
\usepackage[english]{babel}
\usepackage[inner=3cm,outer=3cm,bottom=5cm]{geometry}
\usepackage{setspace}
\usepackage{mathrsfs}
\usepackage[all]{xy}
\usepackage{tikz-cd}

\newtheorem{thm}{Theorem}[section]
\newtheorem{cor}[thm]{Corollary}
\newtheorem{lemma}[thm]{Lemma}

\newtheorem{defn}[thm]{Definition}

\theoremstyle{remark}

\theoremstyle{definition}

\newtheorem{rmk}[thm]{Remark}
\newtheorem{exa}[thm]{Example}

\numberwithin{equation}{thm}

\def\beq{\begin{equation}}
\def\eeq{\end{equation}}

\def\ben{\begin{enumerate}}
\def\een{\end{enumerate}}

\usepackage[OT2,T1]{fontenc}
\DeclareSymbolFont{cyrletters}{OT2}{wncyr10}{m}{n}
\SetSymbolFont{cyrletters}{bold}{OT2}{cmr}{bx}{n}
\DeclareMathSymbol{\Dc}{\mathalpha}{cyrletters}{68}

\def\crash#1{}
\def\N{{\mathbb N}}
\def\Z{{\mathbb Z}}
\def\P{{\mathbb P}}
\def\Q{{\mathbb Q}}
\def\R{{\mathbb R}}
\def\C{{\mathbb C}}
\def\A{{\mathbb A}}

\def\H{{\mathbb H}}

\def\l{\left}
\def\r{\right}
\def\[[{\l[\l[}
\def\]]{\r]\r]}

\def\rk{{\rm rk}\;}

\def\ie{\emph{i.e.}\;}
\def\lc{\emph{loc.cit.}\;}

\def\cA{{\mathcal A}}

\def\cM{{\mathcal M}}

\def\cO{{\mathcal O}}

\def\cS{{\mathcal S}}
\def\cT{{\mathcal T}}

\def\sH{{\mathscr H}}

\def\fC{{\mathfrak C}}

\def\fU{{\mathfrak U}}

\def\fs{{\mathfrak s}}
\def\fri{{\mathfrak i}}
\def\wtilde{\widetilde}

\def\Gal{{\rm Gal\,}}

\def\id{{\rm id\,}}

\def\iso{\xrightarrow{\ \sim\ }}

\def\trdeg{{\text{tr.deg}}}

\def\eps{\epsilon}
\def\vt{{\vec{t}}}
\def\vv{{\vec{v}}}

\def\pf{{\mathbbm{p}}}
\def\b{{\mathbbm{b}}}

\def\sco{{\cO^{_S}}}
\def\tco{{\cO^{_T}}}

\def\sd{{D^{_S}}}
\def\td{{D^{_T}}}

\def\ta{{A^{_T}}}
\newcommand{\norneq}{\mathrel{\ooalign{$\lneq$\cr\raise.22ex\hbox{$\lhd$}\cr}}}

\begin{document}
\setlength{\baselineskip}{0.55cm}	
\title[Canonical factorizations]{Canonical factorizations of morphisms of Berkovich curves}
\author{Velibor Bojkovi\'c}
\address{Laboratoire de Math\'ematiques Nicolas Oresme, Caen, France}
\email{velibor.bojkovic@unicaen.fr}

\keywords{Herbrand extensions, Factorization of morphisms, Berkovich curves, harmonicity properties of morphisms}

\begin{abstract}
We prove that, for certain extensions of valued fields which admit a sensible theory of ramification groups, there exist canonical towers that correspond to the break-points of their Herbrand function. In particular, each of the intermediate field extensions in the tower has a Herbrand function with only one break-point and there is at most one extension with trivial Herbrand function.

We apply the result to the setting of finite morphisms of Berkovich curves where we prove the existence of canonical local and global factorization of such morphisms according to their metric properties. 

Finally, we use the canonical factorizations to prove harmonicity properties finite morphisms satisfy at each type  2 point: formulas that can be regarded as a refinement of the Riemann-Hurwitz formula for such morphisms.
\end{abstract}
\maketitle
\tableofcontents

\section{Introduction}
The present article is a result of an attempt to explore the deep relation between ramification properties of finite extensions of certain valued fields and the metric properties of finite morphisms of analytic curves where such fields appear naturally as the (completed) residue fields of their points, a relation that was recently discovered by M. Temkin in \cite{TeMu}.

To say more, let $k$ be a complete, algebraically closed, complete with respect to a non-trivial and non-archimedean valuation, and (for simplicity) of characteristic 0 (\emph{e.g.} the field $\C_p$). We recall that in the Berkovich approach to analytic geometry over $k$ (which we adopt in this paper), the role of classic Riemann surfaces is taken by quasi-smooth $k$-analytic curves, which loosely can be described by the fact that every $k$-point has a neighborhood isomorphic to a Berkovich open disc over $k$ (we note that all smooth projective $k$-analytic curves, i.e. the analytifications of smooth projective $k$-algebraic curves are quasi-smooth). The structure of such curves is most easily described by the semistable reduction theorem, or better to say, its $k$-analytic avatar, the semistable skeleton. Namely, if $X$ is a quasi-smooth $k$-analytic curve, a semistable skeleton of $X$ is a closed subset $\Gamma$ which, roughly speaking, has a structure of a topological graph and is such that  $X\setminus\Gamma$ is a disjoint union of open discs. For a comprehensive study of quasi-smooth $k$-analytic curves and proof that they admit semistable skeleta one may refer to \cite{Duc-book}.

If now $f:Y\to X$ is a finite morphism of quasi-smooth $k$-analytic curves, it follows from (a suitable interpretation of ) a result of R. Coleman \cite{Col03} that $f$ admits a simultaneous semistable reduction, or that it has a semistable skeleton, meaning that there exists semistable  skeleta $\Gamma_Y$ and $\Gamma_X$ of $Y$ and $X$, respectively, such that $\Gamma_Y=f^{-1}(\Gamma_X)$. As a consequence, for each connected component/disc $D$ in $Y\setminus\Gamma_Y$, the restriction $f_{|D}$ is a finite morphism of open disc and $f(D)$ is a connected component of $X\setminus \Gamma_X$. 

A more subtle notion of a skeleton of a morphism $f$, the so-called radializing skeleton, was introduced in \cite{TeMu}. Following \lc we say that a finite morphism of open unit discs $g:D_1\to D_2$ is radial if for any choice of coordinates $T$ and $S$ on $D_1$ and $D_2$ respectively ( still identifying them with unit discs) and such that $T=0$ is sent to $S=0$ (we say that $T$ and $S$ are compatible), the valuation polygon of the corresponding expansion of $f$ in coordinates $T$ and $S$ does not depend on $T$ and $S$. Then, if $T$ and $S$ are one such pair of coordinates, and if $S=g(T)=\sum_{i=1}^\infty g_i\cdot T^i$ is the coordinate representation of the morphism $f$, then if $f$ is radial the function 
$$
[0,1]\ni r\mapsto\max_{i=1,..,\infty}\{|g_i|\cdot r^i\}\in [0,1] 
$$
does not depend on the compatible coordinates $T$ and $S$ and is called the profile of $f$. We denote it by $\pf_f$. One may notice that the profile is nothing but the multiplicative version of the classical valuation polygon of $g(T)$ and as such enjoys similar properties. For example, since $f$ is a finite morphism the function $g(T)$ has finitely many zeroes and consequently its valuation polygon has finitely many break-points. Then, one deduces that the profile of $f$ has finitely many break-points as well, where the latter are interpreted as the points in $(0,1)$ where the profile is not smooth. One further shows that on the intervals $I\subseteq [0,1]$ that do not contain the break-points, the profile is of the form $r\mapsto a\cdot r^{p^{\alpha}}$, for $r\in I$, where $p$ is the residual characteristic of $k$, $\alpha$ is a non-negative integer and $a\in |k^*|$. We will refer to $p^{\alpha}$ as the local degree of $\pf_{f}$ (over $I$). 

More generally, if $\Gamma=(\Gamma_Y,\Gamma_X)$ is a skeleton of $f$, we say that $f$ is radial with respect to $\Gamma$ (or that $\Gamma$ is radializing for $f$) if for any connected component (open disc) $D\in Y\setminus \Gamma_Y$, the morphism $f_{|D}:D\to f(D)$ is a radial morphism of open unit discs and the profile of $f_{|D}$ depends only on the point on $\Gamma_Y$ to which $D$ is attached (that is, the point in the closure of $D$ in $Y$ that does not belong to $D$). Two remarkable results have been proved in \cite{TeMu}. The first one is that, for any finite morphism $f:Y\to X$ of quasi-smooth $k$-analytic curves, and $\Gamma=(\Gamma_Y,\Gamma_X)$ a semistable skeleton of $f$, there exists a semistable skeleton $(\Gamma_Y',\Gamma_X')$ such that $\Gamma_Y\subset\Gamma_Y'$ and $\Gamma_X\subset\Gamma_X'$, and such that $f$ is radial with respect to $(\Gamma_Y',\Gamma_X')$. In particular, since we may find $\Gamma_Y$ to contain any type 2 point $y$ in $Y$ (in Berkovich classification of points in a quasi-smooth $k$-analytic curve), one can assign to $y$ a profile function $\pf_{f,y}$ which is nothing but the profile of $f_{|D}$, where $D$ is any disc in $Y\setminus \Gamma_Y'$ that is attached to $y$. For the second result, M. Temkin introduces the \emph{almost monogeneous} fields (see Section \ref{sec: almost monogeneous} for the precise definition), for which he constructs a reasonable theory of ramification groups. Namely, if $L/K$ is a finite Galois extension of almost monogeneous fields with Galois group $G$, then one introduces the intertia function $i_G:G\to [0,1]$ with $\sigma\mapsto i_G(\sigma):=\sup_{x\in L^\circ}|\sigma(x)-x|$, where $|\cdot|$ is the norm of $L$ and $L^\circ:=\{x\in L\mid |x|\leq 1\}$, and finally the Herbrand function 
$$
\H_{L/K}: [0,1]\to [0,1],\quad r\mapsto \H_{L/K}(r):=\prod_{\sigma\in G}\max(r,i_G(\sigma)).
$$
Then, one shows that: 1) If $H$ is a normal subgroup of $G$ and $F:=L^H$, then $\H_{L/K}=\H_{F/K}\circ \H_{L/F}$ and 2) For every normal extension $L/K$ of almost monogeneous fields one still can define the Herbrand function by simply putting $\H_{L/K}:=\H_{M/K}\circ\H_{M/L}^{-1}$, where $M$ is the Galois closure of $L/K$, and the point being that Herbrand function still continues to be transitive in the extensions. The author in \lc proceeds by showing that completed residue fields of points in a quasi-smooth $k$-analytic curves are almost monogeneous and that in the setting $f:Y\to X$, $x=f(y)$, above we have equality of functions
$$
\H_{\sH(y)/\sH(x)}\equiv \pf_{f,y}.
$$

In the present article we introduce a class of valued fields which we name Herbrand and whose main property is that the Herbrand function is transitive in the extensions (that is, the properties 1) and 2) above are satisfied, so in particular they contain the class of almost monogeneous fields, see Definition \ref{def: Herbrand}). Their study is the subject of Section 1, and the main result is that if $L/K$ is a finite extension of Herbrand fields then there exists a canonical tower 
\begin{equation}\label{eq: intro tower}
L=L_1/\dots/L_{n+1}=K'/K,
\end{equation}
such that the Herbrand function of each intermediate extension $L_i/L_{i+1}$, for $i=1,\dots,n$ has exactly one break-point, while the Herbrand function of $K'/K$ is trivial (that is identity) and moreover  $\H_{L_i/L_{i+1}}(b_i)<b_{i+1}$ where $b_i$ is the break point of $\H_{L_i/L_{i+1}}$. It is this last condition that gives us the uniqueness of the tower (see Theorem \ref{thm: canonical towers} for precise statement). 

To apply the previous result in the context of finite morphisms of Berkovich curves, in Section 2 we introduce a class of $n$-radial morphisms, which are, sort to speak, a generalization of radial ones, where we require only part of the valuation polygon to be fixed when changing coordinates. More precisely, we say that a finite morphism $f:Y\to X$ of open (unit) discs is weakly $n$-radial if for any choice of compatible coordinates  $S$ and $T$ on $Y$ and $X$, respectively, the corresponding valuation polygon of $f$ expressed in coordinates $T$ and $S$, has the first $n$-slopes and $(n-1)$ break points the same (then, $f$ would be radial if all the slopes and all the break points are the same). We show that they still satisfy nice arithmetic properties (resembling those of the radial ones) and it turns out in Section 3 that their suitable extension to finite morphisms of quasi-smooth $k$-analytic curves provide the right class of morphisms that canonically factorize. 

To this end, let $f:Y\to X$ be a finite morphism of quasi-smooth $k$-analytic curves and let $\Gamma=(\Gamma_Y,\Gamma_X)$ be its skeleton. We recall that for any type 2 point $y$ in a curve $Y$, and its image $x\in X$, the morphism $f$ induces a finite extension of completed residue fields $\sH(y)/\sH(x)$. This in turn induces a finite extension of the residue fields $\wtilde{\sH(y)}/\wtilde{\sH(x)}$. Then, we say that $f$ is residually separable (resp. purely inseparable) at $y$ if the latter extension of fields is so and these notions extend to points of type 3 and 4 by a suitable extension of scalars. The separable degree of $f$ at $y$ is then the degree of the separable closure of $\wtilde{\sH(x)}$ in $\wtilde{\sH(y)}$ over $\wtilde{\sH(x)}$. Finally, we say that $f$ is uniformly residually separable with respect to $\Gamma$ if outside of $\Gamma_Y$, $f$ is an isomorphism while the separable degree of $f$ at any point in $\Gamma_Y$ is the same. We say that $f$ is uniformly weakly $n$-radial with respect to $\Gamma$ if the following conditions are satisfied: 1) The separable degree of $f$ at $y$ is the same for all the points $y\in\Gamma_Y$ and 2) for every open disc $D$ attached to $\Gamma_Y$, $f_{|D}$ is weakly $n$-radial with first $n$ slopes not depending on $D$ (see Definition \ref{def: uniform n radial}). Similarly, we define uniformly radial morphisms (with respect to some skeleton) by suitably modifying condition 2) above.

Now we may state our local and global factorization theorems, which are the subject of Section 3. Namely, by using the above equality of the profile and Herbrand functions, canonical towers in Herbrand extensions and Berkovich theorem \ref{thm: Berk thm}, we prove that if $f:Y\to X$ is a finite morphism of quasi-smooth $k$-analytic curves and $y\in Y$ is a point (not in $Y(k)$), then locally around $y$, the morphism canonically decomposes into factors that have profiles (at corresponding points, that is, images of $y$) with exactly one break-point and at most one profile that is trivial. Furthermore, the tower for the extension $\sH(y)/\sH(x)$ that is induces by the factorization is precisely the one we described in \eqref{eq: intro tower} (see Theorem \ref{thm: local fact}). Any two such factorizations are locally isomorphic and we call (any) such obtained factorization a canonical factorization of $f$ at $y$.

In the global version of this result, Theorem \ref{thm: global fact},  we show that if $f$ is a uniformly weakly $n$-radial morphism with respect to $\Gamma$, then we can decompose it as
$$
Y\xrightarrow[]{g}Z\xrightarrow{h}Z'\xrightarrow{h'}X
$$
where $Z$ and $Z'$ are quasi-smooth $k$-analytic curves, $g$ is uniformly weakly $(n-1)$-radial with respect to $(\Gamma_Y,g(\Gamma_Y))$, $h$ is uniformly radial with respect to $(g(\Gamma_Y),h\circ g(\Gamma_Y))$ with property that for each point $z$ in $g(\Gamma_Y)$, the profile $\pf_{h,z}$ has exactly one break point and finally, $h'$ is uniformly residually separable with respect to $(h\circ g(\Gamma_Y),\Gamma_X)$. Furthermore, for any $y\in\Gamma_Y$, a canonical factorization of $g$ at $y$ ``composed'' with $h$ and $h'$ gives us a canonical decomposition of $f$ at $y$. As a consequence of this result, in the  Corollary \ref{cor: global fact} we treat uniformly radial morphisms and show that if $f$ is uniformly radial with respect to $\Gamma$, then we can decompose it into series of finite morphisms of quasi-smooth $k$-analytic curves such that for any $y\in\Gamma_Y$ the decomposition is canonical decomposition of $f$ at $y$.

Finally, in Section 4 we put to use the results of Section 3 by studying harmonicity properties of $f$ at points of type 2. Roughly speaking this part of the article goes as follows. We keep the setting $f:Y\to X$ as above and let $y\in Y$ be a point of type 2 in the interior of $Y$ (that is, $y$ has a smooth neighborhood in $Y$). Let $\pf_{f,y}$ be the profile of $f$ at $y$, let $0<b_1<\dots<b_n<1$ be its break points and suppose (for this introduction) that $n\geq 1$. Let us put $b_0:=0$ and $b_{n+1}:=1$ and let $p^{\alpha_i}$ be the local degree of $\pf_{f,y}$ over $(b_i,b_{i+1})$, for $i=0,\dots,n$. Since $y$ is of type 2 the field $\wtilde{\sH(y)}$ is a function field of a smooth projective $\wtilde{k}$-algebraic curve (where $\wtilde{k}$ is the residue field of $k$) of genus $g$ and we put $g(y):=g$. Furthermore, $f$ admits a neighborhood $U$ in $Y$ such that $U\setminus\{y\}$ is a disjoint union of open discs and finitely many open annuli in $Y$ and we denote their set by $T_yY$. For any $\vt\in T_yY$ let $A_{\vt}$ be a small enough open annulus  contained in $\vt$ such that the closure of $A_\vt$ in $Y$ contains $y$, $f_{|A_\vt}:A_{\vt}\to f(A_\vt)$ is a finite \'etale morphism of open annuli which is uniformly radial with respect to the skeleton of $f_{|A_\vt}$ coming from skeleta of $A_\vt$ and $f(A_\vt)$ (the latter are simply the complements of the union of all the open discs in $A_\vt$ and $f(A_\vt)$, respectively). Let further $T_\vt$ be a coordinate on $A_\vt$ that identifies it with an open annulus $A(0;r_\vt,1)$ of inner radius $r_\vt$ and outer radius $1$ and such that the points $\eta_{\vt,\rho}$, $\rho\in (r_\vt,1)$, converge to $y$ as $\rho$ goes to 1, where $\eta_{\vt,\rho}$ is the unique point on the skeleton of $A(0;r_\vt,1)$ that is of radius $\rho$ (see the end of Section \ref{sec: disc and annuli} for the definition of radius). Let $\pf_{f,\eta_{\vt,\rho}}$ be the profile of $f$ at $\eta_{\vt,\rho}$, let $0<b_{\vt,1}(\rho)<\dots<b_{\vt,n(\vt)}(\rho)<1$ be the break points of $\pf_{f,\eta_{\vt,\rho}}$ and let us put $b_{\vt,0}(\rho):=0$ and $b_{\vt,n(\vt)+1}(\rho)=1$. Let further $p^{\alpha_{\vt,i}}$ be the local degree of $\pf_{f,\eta_{\vt,\rho}}$ over the interval $(b_{\vt,i}(\rho),b_{\vt,i+1}(\rho))$, for $i=0,\dots,n(\vt)$. Here we note that $n(\vt)\geq n$ because of the continuity of the profile function and that the local degrees will not depend on $\rho$ if we choose $A_\vt$ sufficiently small. Finally, we fix an $i\in \{1,\dots,n\}$ and let $I_{\vt,i}$ denote the subset of $\{1,\dots,n(\vt)\}$ such that for every $j\in I_{\vt,i}$ we have $\lim(b_{\vt,j}(\rho))=b_i$, as $\rho\to 1$. Then, we have the following formula
$$
(p^{\alpha_i}-p^{\alpha_{i-1}})\cdot (2-2\cdot g(y))=\sum_{\vt\in TyY}\left(\sum_{j\in I_{\vt,i}}(p^{\alpha_{\vt,j}}-p^{\alpha_{\vt,j-1}})\cdot \big(\partial_{\vt}b_{\vt,j}-1\big)\right),
$$
where $\partial_{\vt}b_{\vt,j}:=\lim_{\rho\to 1}\frac{d\log b_{\vt,j}(\rho)}{d\log \rho}$. This is an instance of the formulas in Theorem \ref{thm: loc harmonicity} and which refine the local Riemann-Hurwitz formula for $f$ at $y$, Theorem \ref{thm: RH}.

\section{Canonical towers in Herbrand extensions}

\subsection{The set $\Lambda_p$}

\begin{defn}\label{defn: lambda} Let $p$ be a prime number. We denote by $\Lambda_p$ the set of functions $f:[0,1]\to [0,1]$ that satisfy the following properties:
\begin{enumerate}
 \item The function $f$ is continuous, strictly increasing with $f(0)=0$ and $f(1)=1$.
\item There exists real numbers $b_0=0<b_1<\dots<b_n<b_{n+1}=1$, positive real numbers $a_1,\dots,a_{n+1}$ and non-negative integers $\alpha_0,\dots,\alpha_n$ such that $f$ restricted to $(b_{i-1},b_i)$ is of the form $r\mapsto a_i\cdot r^{p^{\alpha_{i-1}}}$, for $i=1,\dots,n+1$. 

 \item For each $i=0,\dots,n$, as in the previous point, we have $\alpha_i<\alpha_{i+1}$ with $\alpha_0=0$.
\end{enumerate}
We will refer to the points $b_1,\dots,b_n$ as the \emph{break-points} of $f$, to the numbers $p^{\alpha_i}$ as the \emph{local degrees} of $f$ while $a_i$ are the \emph{local coefficients} of $f$. We call $p^{\alpha_{n}}$ the \emph{degree} of $f$. 
\end{defn}

 One may note that if in the previous definition $n=0$ for some $f\in \Lambda_p$, then $f$ is the identity function. Here are some further easily established properties of the set $\Lambda_p$. 
\begin{lemma}\label{lem: lambda prop}
 \begin{enumerate}
 \item  The set $\Lambda_p$ is closed under function-composition of its elements.
  \item For any $f\in\Lambda_p$, $a_{n+1}=1$.
  \item If $f\in \Lambda_p$, then $f^{-1}:[0,1]\to [0,1]$ is a continuous, strictly increasing function (with $f^{-1}(0)=0$ and $f^{-1}(1)=1$).
  \item Keeping the notation from the definition, each $f\in \Lambda_p$ is given by $r\mapsto \max\limits_{i=1,\dots,n+1}\{a_i\,r^{p^{\alpha_{i-1}}}\}$, for $r\in[0,1]$. In particular, $f$ is a convex function. 
  \item If the number of break-points of $f\in\Lambda_p$ is $n\geq 1$ and keeping the notation as in Definition \ref{defn: lambda}, then for $i=1,\dots,n+1$, we have
 \begin{equation}\label{eq: ai bi}
 a_i=\prod_{j=i}^{n+1}b_j^{p^{\alpha_j}-p^{\alpha_{j-1}}}.
 \end{equation}
  \end{enumerate}
\end{lemma}
\begin{proof}
 The first three points are straightforward, so we address the fourth one. Let $f\in \Lambda_p$. If $f$ has no break-points then it is the identity function so there is nothing to prove. Otherwise, suppose that $f$ has $n\geq 1$ break-points and let them be $b_1<\dots<b_n$ (we put as usually $b_0=0$, $b_{n+1}=1$ and use notation from Definition \ref{defn: lambda}). Then, for each $i=1,\dots,n$, we have $a_i\cdot b_i^{p^{\alpha_{i-1}}}=a_{i+1}\cdot b_i^{p^{\alpha_{i}}}$, by continuity of $f$ at $b_i$, so in particular  
 \begin{equation}\label{eq: ai bi 1}
 a_i=a_{i+1}\cdot b_i^{p^{\alpha_{i}}-p^{\alpha_{i-1}}}.
 \end{equation}
 Then, for $r\in [0,1]$ we obtain 
 \begin{align*}
\text{for } r<b_i,&\quad a_i\cdot r^{p^{\alpha_{i-1}}}=a_{i+1}\cdot b_i^{p^{\alpha_{i}}-p^{\alpha_{i-1}}}\cdot r^{p^{\alpha_{i-1}}}>a_{i+1}\cdot r^{p^{\alpha_{i}}}\\
\text{for } r>b_i,&\quad a_i\cdot r^{p^{\alpha_{i-1}}}=a_{i+1}\cdot b_i^{p^{\alpha_{i}}-p^{\alpha_{i-1}}}\cdot r^{p^{\alpha_{i-1}}}<a_{i+1}\cdot r^{p^{\alpha_{i}}}.
\end{align*}
It follows that $f$ is given by $r\mapsto \max\{a_i\,r^{p^{\alpha_{i-1}}}\mid i=1,\dots,n+1\}$, and consequently it is convex.

For the point {\emph (5)}, we continue the chain of equations in \eqref{eq: ai bi 1} which together with $a_{n}\,b_n^{p^{\alpha_{n-1}}}=b_n^{p^{\alpha_{n}}}$ gives us \eqref{eq: ai bi}.
\end{proof}

\begin{defn}
 We say that $f\in \Lambda_p$ is simple if it has exactly one break-point. We will sometimes also denote the break-point by $b_f$, while the nontrivial local coefficient by $a_f$. 
\end{defn}

In particular, if $b_f$ and $p^\alpha$ are the break-point and degree, respectively, of a simple function $f$, then $f$ is given by
$$
f(r)=\begin{cases}
          b_f^{p^\alpha-1}\cdot r,\quad r\in [0,b_f]\\
          r^{p^\alpha},\quad r\in [b_f,1].
         \end{cases}
$$

\begin{lemma}\label{lem: factoring in lambda}
 Let $f\in \Lambda_p$ and suppose that it has $n\geq 2$ break-points. Then, there exist unique simple functions $f_1,\dots,f_n\in \Lambda_p$ such that 
 \begin{enumerate}
  \item We have $f=f_n\circ\dots\circ f_1$;
  \item For $i=1,\dots, n$ we have $f_i(b_{f_i})<b_{f_{i+1}}$.
 \end{enumerate}
 Explicitly, keeping the notation from Definition \ref{defn: lambda}, the functions $f_i$ are given by
 \begin{equation}\label{eq: factor profile}
 f_i(r)=\begin{cases}
         b_i^{p^{\alpha_{i}}-p^{\alpha_{i-1}}}\cdot r,\quad r\in[0,b_i^{p^{\alpha_{i-1}}}]\\
         r^{p^{\alpha_{i}-\alpha_{i-1}}},\quad r\in[b_i^{p^{\alpha_{i-1}}},1].
        \end{cases} 
 \end{equation}
\end{lemma}
\begin{proof}
  We keep the notation from Definition \ref{defn: lambda} and we also put $b_0:=0$ and $b_{n+1}:=1$. 
 
 \emph{Existence.} Clearly, $f_i\in\Lambda_p$, $f_i$ is simple with the break-point $b_{f_i}:=b_i^{p^{\alpha_{i-1}}}$ and we also note that 
 \begin{equation}\label{eq: passage 1}
 f_i(b_{f_i})=b_i^{p^{\alpha_{i}}}<b_{i+1}^{p^{\alpha_{i}}}=b_{f_{i+1}}.
 \end{equation}
 Next we show that $f=f_n\circ\dots\circ f_1$. Let $r\in [b_{i-1},b_{i}]$, for some $i=1,\dots,n+1$. Then, by constructions of $f_i$'s, we have that $f_{i-1}\circ\dots\circ f_1(r)=r^{p^{\alpha_{i-1}}}$. On the other side, $f_i(r^{p^{\alpha_{i-1}}})=b_i^{p^{\alpha_{i}}-p^{\alpha_{i-1}}}\cdot r^{p^{\alpha_{i-1}}}<b_{f_{i+1}}$, so we obtain inductively 
 $$
 f_n\circ\dots\circ f_i(r^{p^{\alpha_i}})=\big(\prod_{j=i}^nb_j^{p^{\alpha_{j}}-p^{\alpha_{j-1}}}\big)\cdot r^{p^{\alpha_{i-1}}}=a_i\cdot r^{p^{\alpha_{i-1}}}=f(r),
 $$
 by Lemma \ref{lem: lambda prop}.

 \emph{Uniqueness.} Suppose we have a factorization $f=f'_{l}\circ\dots\circ f'_n$ satisfying the conditions of the theorem. It is enough to prove that $f'_1$ is uniquely determined (that is $f'_1=f_1$) as then the claim for the rest of the functions will follow by applying the same reasoning and constructions of the simple factors to the function $f'_n\circ\dots\circ f'_2$ instead of $f$. Let $d_1$ be the degree of $f'_1$. Since $f'_1(b_{f'_1})<b_{f'_2}$ and consequently for each $i=2,\dots,n$, $f'_{i}\circ\dots\circ f'_1(b_{f_1})<b_{f'_{i+1}}$, we have that for $r\in [0,b_{f'_1}]$, $f(r)=a_{f'_l}\dots a_{f'_2}\cdot r$.  On the other side, for $r\in[b_{f'_1},{f'}_1^{-1}(b_{f'_2}))$ we have $f'_1(r)=r^{d_1}<b_{f'_2}$ so, consequently, for each $i=2,\dots,n$, we have $f'_i\circ\dots\circ f'_1(r)=a_{f'_l}\,\dots\, a_{f'_2}\cdot r^{d_1}$. It follows that in this case $f(r)=a_{f'_l}\,\dots\,a_{f'_2}\cdot r^{d_1}$, hence $b_{f'_1}$ is the smallest break-point of $f$, that is $b_{f'_1}=b_1=b_{f_1}$ and $d_1=p^{\alpha_1}$. So $f'_1=f_1$ and the claim follows. 
 \end{proof}

 \begin{defn}\label{def: lambda decomposition}
  Suppose $f\in\Lambda_p$ has more than one break-point. We call the factorization $f=f_n\circ\dots\circ f_1$ from Lemma \ref{lem: factoring in lambda} the \emph{canonical factorization of $f$} (into simple factors). 
  \end{defn}
\subsection{Herbrand extensions}

\subsubsection{} 
By a valued field we mean a field equipped with a nontrivial, nonarchimedean norm. If $K$ is a valued field with norm $|\cdot|$, we denote by $K^{\circ}$ and $K^{\circ \circ}$ its ring of integers and the corresponding maximal ideal, respectively. Its residue field is denoted by $\wtilde{K}$.

We recall that a valued field $K$ is said to be Henselian if for every finite field extension $L/K$, the norm $K$ uniquely extends to a norm on $L$. 

Let $L/K$ be a finite, Galois extension of valued fields (meaning that the norm of $L$ on $K$ coincides with the norm of $K$), with Galois group $G$. 

The \emph{inertia function} $i_G:=i_{G,L/K}:G\to [0,1]$ is defined as 
$$
\sigma\in G\mapsto i_G(\sigma)=\sup_{t\in L^\circ}|\sigma(t)-t|. 
$$
A point $r\in [0,1]$ is called \emph{a break-point for $i$} if there exists a $\sigma\in G$ with $i_G(\sigma)=r$. If a break-point is different from 0 or 1 we will call it nontrivial. Let $0=r_0<r_1<\dots<r_n\leq1$  be all the breakpoints of $i$. The \emph{ramification filtration of $G$} assigned to the inertia function is given by the subgroups $G_{r_i}:=\{\sigma\in G\mid i_G(\sigma)\leq r_i\}$. Clearly we have
$$
\{\id\}=G_{r_0}\norneq G_{r_1}\norneq\dots\norneq G_{r_n}=G.
$$
\begin{lemma}\label{lem: p-groups}
 Suppose that $\text{char}(\wtilde{K})=p>0$. If for some $i=1,\dots,n$, $r_i<1$, then the group $G_{r_i}$ is a $p$-group and $|G_{r_i}|$ is a power of $p$. 
\end{lemma}
\begin{proof}
 This is standard. It is enough to prove that if $\sigma\in G_{r_i}$, then $\sigma^p\in G_{r_{i-1}}$, where $i=1,\dots,n$. 
 
 Let $t\in L^\circ$. Then
 \begin{align}
  |\sigma^p(t)-t|&=|\sigma\big(\sigma^{p-1}(t)+\dots+\sigma(t)+t\big)-\big(\sigma^{p-1}(t)+\dots+\sigma(t)+t\big)|\nonumber\\
  &\leq i_G(\sigma)\cdot|\sigma^{p-1}(t)+\dots+\sigma(t)+\id(t)|\label{eq: bound}
 \end{align}
Finally, we note that for $\sigma_1,\dots,\sigma_p\in G_{r_i}$, and $t\in L^\circ$ we have
$$
|\sigma_1(t)+\dots+\sigma_p(t)|=|(\sigma_1(t)-t)+\dots+(\sigma_p(t)-t)+p\cdot t|\leq \max\{r_i,p\}\cdot|t|\leq \max\{r_i,p\},
$$
which combined with \eqref{eq: bound} yields the result.
\end{proof}

We call the function $\H:=\H_{L/K}:[0,1]\to [0,1]$, defined as 
\begin{equation}\label{eq: H def}
r\in[0,1]\mapsto \prod_{\sigma\in G}\max\{r,i_G(\sigma)\},
\end{equation}
the \emph{Herbrand function} of the Galois extension $L/K$. It follows that the function $\H$ is a continuous, piece-wise monomial function, that is, outside of finitely many points of the interval $(0,1)$ it is of the form $r\mapsto a\cdot r^d$, where $a\in (0,1)$. The points of the interval $(0,1)$ where it fails to be smooth are precisely the break-points of the inertia function (different from 0 and 1). We will refer to them as the nontrivial break-points of $\H$, while we will call $0$ and $1$ the trivial ones.

\subsubsection{}
\begin{defn}\label{def: Herbrand}
 Let $K$ be a henselian field. We say that a finite Galois extension $L/K$ is a \emph{Herbrand extension}, if for every normal subgroup $H\trianglelefteq G$ and corresponding intermediate field $F:=L^H$, we have
 \begin{equation}\label{eq: transitivity}
  \H_{L/K}=\H_{F/K}\circ \H_{L/F}.
 \end{equation}
We say that a finite extension $L/K$ is a \emph{Herbrand extension} if its Galois closure is so.
Finally, if every finite extension of $K$ is Herbrand, we say that $K$ is a \emph{Herbrand field}.
\end{defn}
If $L/K$ is a Herbrand extension, then we define $\H_{L/K}:=\H_{M/K}\circ \H_{M/L}^{-1}$, where $M$ is a finite Galois extension of $L/K$ (point being that if $L/K$ were Galois, then the two constructions of $\H_{L/K}$ coincide, thanks to \eqref{eq: transitivity}).
\begin{rmk}
 The motivation for the name Herbrand extension comes from the classical theory of ramification, namely from Herbrand theorem on behavior of ramification groups in upper and lower indexing. Crucial result used in the proof of Herbrand theorem is exactly the transitivity property \eqref{eq: transitivity} (See \cite[Chapter IV, Section 3.]{SerreLF}).
\end{rmk}

\subsubsection{}\label{sec: almost monogeneous} 
Here is an important example of Herbrand fields. Recall that a henselian extension of valued fields $L/K$ is said to be \emph{monogeneous} if there exists $x\in L^\circ$ such that $L^\circ=K^\circ[x]$. Following \cite[Section 4.]{TeMu} we say that it is \emph{almost monogeneous} if for every $r<1$ there exists an $a_r\in K^\circ$ with $r\leq |a_r|$ and $x_r\in L^\circ$ such that $a_rL^\circ\subseteq K^\circ[x_r]$. Consequently, $L=K(x_r)$ and $L/K$ is finite. We note that monogeneous extensions of valued fields are almost monogeneous. 

A Henselian valued field $K$ is called \emph{almost monogeneous} if its every finite extension is almost monogeneous. Then, \cite[Theorem 4.3.5.]{TeMu} proves that almost monogeneous fields are Herbrand. 

\subsubsection{}
The motivation to introduce the set $\Lambda_p$ is the following lemma.
\begin{lemma}\label{lem: H in lambda}
 Let $K$ be a Herbrand field and let $L/K$ be a finite field extension. If $p>0$ is the residual characteristic of $K$, then $\H_{L/K}\in \Lambda_p$.
\end{lemma}
\begin{proof}
 If $L/K$ is in addition Galois, then the result follows from Lemma \ref{lem: p-groups}. More precisely, as we already remarked, $\H_{L/K}$ is continuous, strictly increasing and surjective function $[0,1]\to [0,1]$ which has finitely many break-points, so condition \emph{(1)} in Definition \ref{defn: lambda} is satisfied. Let $b<c$ be two consecutive break-points of $\H_{L/K}$ and $r\in (b,c)$. Then, 
 \begin{equation}\label{eq: def group ord}
 \H_{L/K}(r)=\prod_{\sigma\in G}\max\{r,i(\sigma)\}=\big(\prod_{\sigma\in G_b} r\big)\cdot\big(\prod_{\sigma\in G\setminus G_b}i(\sigma)\big)=a\cdot r^{|G_b|},
 \end{equation}
 so by Lemma \ref{lem: p-groups}, $\H_{L/K}$ also satisfies conditions \emph{{(2)}} and \emph{(3)} as well.
 
 Suppose now that $L/K$ is not Galois, and let $M$ be its Galois closure, with $G=Gal(M/K)$ and $H=Gal(M/L)$. By definition, $\H_{L/K}=\H_{M/K}\circ \H^{-1}_{M/L}$, and the latter is clearly a continuous, surjective function on the segment $[0,1]$. 
 
 We also notice that the set of the break-points of $\H_{M/L}$ is a subset of the set of break-points of $\H_{M/K}$. So let $h_0=0<h_1<\dots< h_m\leq h_{m+1}=1$ be the break-points of $\H_{M/L}$, and let $h_i=h_{i,1}<\dots< h_{i,j(i)}<h_{i,j(i)+1}=h_{i+1}$ be the break-points of $\H_{M/K}$ in the interval $[h_i,h_{i+1}]$, $i=0,\dots,m$. Then, if we denote by $h'_{i,j}$ the image of $h_{i,j}$ by $\H_{M/L}$, it is not difficult to see that the break-points (as the points where the function is not smooth) of $\H_{L/K}$ are precisely $h'_{i,j}$, $i=0,\dots,m$, $j=1,\dots,j(i)+1$. Let $r\mapsto a_{i,j}\cdot r^{p^{\alpha_{i,j}}}$ (resp. $r\mapsto a_i\cdot r^{p^{\alpha_i}}$) be the local restriction of $\H_{M/K}$ (resp. $\H_{M/L}$) over the interval $(h_{i,j},h_{i,j+1})$ (resp. $(h_i,h_{i+1})$). The equation \eqref{eq: def group ord} gives us that $p^{\alpha_{i,j}}$ (resp. $p^{\alpha_i}$) is nothing but $|G_{h_{i,j}}|$ (resp. $|H_{h_i}|$). We next note that 
 $$
 H_{h_i}=G_{h_i}\cap H=G_{h_{i,1}}\cap H=\dots =G_{h_{i,j(i)}}\cap H,
 $$
 which together with Second isomorphism theorem for groups implies 
 $$
\frac{HG_{h_{i,j}}}{G_{h_{i,j}}}\simeq \frac{H}{H_{h_i}},\quad i=0,\dots,m,\, j=1,\dots,j(i),	  
 $$
 so that in particular, 
 $$
\frac{|G_{h_{i,j}}|}{|H_{h_i}|}=\frac{|HG_{h_{i,j}}|}{|H|}, \quad i=0,\dots,m,\, j=1,\dots,j(i).
 $$
 As a direct consequence, we have
 \begin{equation}\label{eq: h in lambda 1}
 \frac{|G_{h_{i,1}}|}{|H_{h_i}|}< \frac{|G_{h_{i,2}}|}{|H_{h_i}|}< \dots < \frac{|G_{h_{i,j(i)}}|}{|H_{h_i}|}, \quad i=1,\dots,m,\, j=1,\dots,j(i)
 \end{equation}
 and 
 \begin{equation}\label{eq: h in lambda 2}
  \frac{|G_{h_{i,j(i)}}|}{|H_{h,i}|}<\frac{|G_{h_{i+1}}|}{|H_{h_{i+1}}|}.
 \end{equation}
Let $r\in (h'_{i,j}, h'_{i,j+1})$ for some $i=0,\dots,m$ and $j=1,\dots,j(i)$ (we note that the union of closures of such intervals covers $[0,1]$). Then, 
\begin{align*}
\H_{L/K}(r)&=\H_{M/K}\big(\H^{-1}_{M/L}(r)\big)=\H_{M/K}\big(\big(\frac{1}{a_i}\big)^{\frac{1}{p^{\alpha_i}}}\cdot r^{\frac{1}{p^{\alpha_i}}}\big)=a_{i,j}\cdot\big(\frac{1}{a_i}\big)^{{p^{\alpha_{i,j}-\alpha_i}}}\cdot  r^{p^{\alpha_{i,j}-\alpha_i}}\\
&=a_{i,j}\cdot\big(\frac{1}{a_i}\big)^{{p^{\alpha_{i,j}-\alpha_i}}}\cdot  r^{\frac{|G_{h_{i,j}}|}{|H_{h_i}|}},
\end{align*}
and conditions \emph{(2)} and \emph{(3)} from Definition \ref{defn: lambda} now follow from equations \eqref{eq: h in lambda 1} and \eqref{eq: h in lambda 2}.
\end{proof}

\begin{defn}
 Let $K$ be a Herbrand field. We say that a finite field extension $L/K$ is  \emph{simply ramified} if the corresponding Herbrand function $\H_{L/K}$ has exactly one break-point and its degree is equal to the degree of extension $|L:K|$.
\end{defn}

Here is the starting point of our factorization results.

\begin{thm}\label{thm: canonical towers}
 Let $K$ be a Herbrand field and let $L/K$ be a finite field extension. Then, there exists a unique tower 
 $$
 L=L_1/L_2/\dots/L_n/L_{n+1}=K'/K
 $$
 such that 
 \begin{enumerate}
  \item The field extension $L_i/L_{i+1}$ is simply ramified for all $i=1,\dots, n$. 
  \item The field extension $K'/K$ has trivial (that is, identity) Herbrand function.
  \item The decomposition $\H_{L/K}=\H_{L_n/L_{n+1}}\circ\dots\circ \H_{L_1/L_2}$ is the canonical decomposition of $\H_{L/K}$ into simple factors (in the sense of Definition \ref{def: lambda decomposition}).
 \end{enumerate}
\end{thm}
\begin{proof}
 First we deal with existence. Let $M/K$ be the Galois closure of $L/K$  and let us put $G=\Gal(M/K)$ and $H=\Gal(M /L)$. Let $r_0=0<r_1<\dots<r_m<r_{m+1}=1$ be the break-points of $\H_{M/K}$. We note that either $G_{r_m}=G$ either $G_{r_m}$ is properly contained in $G$, depending whether there are elements $\sigma\in G$ with $i_G(\sigma)=1$. Since each ramification group $G_{r_i}$ is a normal subgroup of $G$, $H\, G_{r_i}=G_{r_i}\, H$ is a subgroup of $G$. Let us define subgroups 
 $$
 H_1=H\lneq H_2\lneq \dots\lneq H_n\lneq H_{n+1}\leq G,
 $$
 such that $\{H_i\mid i=1,\dots,n+1\}=\{HG_{r_i}\mid i=0,\dots,m\}$. For each $i=1,\dots,n+1$, let $s(i):=\min\{j\mid H_i=H G_{r_j}\}$, so that we have in particular $H_i=H\, G_{r_{s(i)}}$ and $G_{r_{s(i)}}$ is the minimal ramification subgroup with this property. Let us further put $L_i:=M^{H_i}$. We will prove that fields $L_i$ satisfy the conditions of the theorem.

We fix $i=1,\dots,n$ and we prove first that $L_i/L_{i+1}$ is a simply ramified extension. By definition, $\H_{L_i/L_{i+1}}=\H_{M/L_{i+1}}\circ\H_{M/L_i}^{-1}$ and by Lemma \ref{lem: H in lambda} this function is in $\Lambda_p$. Let $t_1<t_2$ be two consecutive break-points of $\H_{M/L_{i+1}}$ (we recall that the break-points of $\H_{M/L_i}$ are contained in the break-points of $\H_{M/L_{i+1}}$ because $H_i\subset H_{i+1}$, and in turn all these are contained in $\{r_0,\dots,r_{m+1}\}$). Then, as follows from definition \eqref{eq: H def}, the local degree of $\H_{M/L_{i+1}}$ (resp. $\H_{M/L_i}$) over the interval $(t_1,t_2)$ is $|G_{t_1}\cap H_{i+1}|$ (resp. $|G_{t_1}\cap H_i|$), so that the local degree of $\H_{L_i/L_{i+1}}$ over the interval $(\H_{M/L_i}(t_1),\H_{M/L_i}(t_2))$ is 
\begin{equation}\label{eq: loc degree}
 \frac{|G_{t_1}\cap H_{i+1}|}{|G_{t_1}\cap H_i|}.
\end{equation}

\emph{Case 1.} $t_1$ is in the set $\{0,r_1,\dots,r_{s(i+1)-1}\}$ (hence $t_2\in\{r_1,\dots,r_{s(i+1)}\}$). Then $G_{t_1}\subset H_i\subset H_{i+1}$, so that \eqref{eq: loc degree} becomes 1. 

\emph{Case 2.} $t_1$ is in the set $\{r_{s(i+1)},\dots,r_m\}$. Then, $G_{r_{s(i+1)}}\leq G_{t_1}$ so  $G_{t_1}\, H_{i+1}=G_{t_1}\, H_i$, therefore
$$
|G_{t_1}H_{i+1}|=\frac{|G_{t_1}|\cdot |H_{i+1}|}{|G_{t_1}\cap H_{i+1}|}=\frac{|G_{t_1}|\cdot|H_i|}{|G_{t_1}\cap H_i|},
$$
which implies
\begin{equation}\label{eq: loc degree 2}
\frac{|G_{t_1}\cap H_{i+1}|}{|G_{t_1}\cap H_i|}=\frac{|H_{i+1}|}{|H_i|}.
\end{equation}
From these two cases it follows that 
\begin{equation}\label{eq: H br point}
\H_{L_i/L_{i+1}}\quad \text{is simple with a break-point in}\quad \H_{M/L_i}(r_{s(i+1)})
\end{equation}
and of degree $|H_{i+1}|/|H_i|$ which is precisely the degree of extension $L_i/L_{i+1}$. We proved that $L_i/L_{i+1}$ is simply ramified. 

For part \emph{(2)}, if $L_{n+1}=K$, or equivalently, if $H_{n+1}=G$, there is nothing to prove. However, if $H_{n+1}\lneq G$, then the corresponding formula \eqref{eq: loc degree} shows that $\H_{K'/K}$ has local degrees equal to 1, since each ramification group $G_{r_0},\dots,G_{r_M}$ is a subgroup of both $H_{n+1}$ and $G$. Hence, $\H_{K'/K}$ is trivial.

To prove part \emph{(3)}, having in mind \emph{(1)}, \emph{(2)} and \eqref{eq: H br point} it is enough to prove that for $i=1,\dots, n$, 
$$
\H_{L_i/L_{i+1}}(\H_{M/L_i}(r_{s(i+1)}))<\H_{M/L_{i+1}}(r_{s(i+2)}),
$$
or, equivalently,
$$
\H_{M/L_{i+1}}(r_{s(i+1)})<\H_{M/L_{i+1}}(r_{s(i+2)}),
$$
which is clear. 

For the uniqueness, let $L=L_1'/L_2'/\dots/L_n'/L_{n+1}'/K$ be another tower satisfying the properties in the theorem, and let $H=H_1'\leq H_2'\leq \dots \leq H_{n+1}'$ be the corresponding subgroups of $G$. We will prove that $H_i'=H_i$ for $i=1,\dots,n+1$.  

For $i=1$ this is clear, so suppose that for some $i_0\in\{1,\dots,n\}$ we have $H_i=H_i'$ for $i=1,\dots,i_0$. Because $\H_{L_{i_0}/L_{i_0+1}}=\H_{L_{i_0}'/L_{i_0+1}'}$ and because of our assumption, \eqref{eq: loc degree} implies that for each $t\in\{r_0,\dots,r_n\}$ we have $|G_{t}\cap H_{i_0+1}|=|G_{t}\cap H_{i_0+1}'|$.
Also, since $|L_{i_0'}:L_{i_0+1}'|=|L_{i_0}:L_{i_0+1}|$ (the two simply ramified extensions have the same Herbrand functions), it follows that 
$$
\frac{|H_{i_0+1}'|}{|H_{i_0}'|}=\frac{|H_{i_0+1}|}{|H_{i_0}|},\quad\text{so}\quad |H_{i_0+1}'|=|H_{i_0+1}|.
$$
But then for $t=s(i_0+1)$, 
$$
|G_{t}H_{i_0+1}'|=\frac{|G_{t}|\cdot|H_{i_0+1}'|}{|G_{t}\cap H_{i_0+1}'|}=\frac{|G_{t}|\cdot|H_{i_0+1}|}{|G_{t}\cap H_{i_0+1}|}=|G_{t}H_{i_0+1}|=|H_{i_0+1}|=|H_{i_0+1}'|.
$$
Finally, this implies that 
$$
H_{i_0+1}=G_tH_{i_0}=G_tH_{i_0}'\leq G_tH_{i_0+1}'=H_{i_0+1}',
$$
and $H_{i_0+1}=H_{i_0+1}'$.
\end{proof}

\begin{rmk}
 \emph{(1)} If $L/K$ from Theorem \ref{thm: canonical towers} is Galois, then the tower obtained in the theorem corresponds to the tower coming from the ramification groups. 
 
 \emph{(2)} It is worth noting that theorem above shows that if $L/K$ has simple $\H_{L/K}$ but is not a simply ramified extension, then it factors canonically as $L/K'/K$, where $L/K$ is simply ramified extension with $\H_{L/K'}=\H_{L/K}$ and $\H_{K'/K}=\id$.
\end{rmk}

\section{$n$-radial morphisms of quasi-smooth $k$-analytic curves}
Although many of the results that follow hold also when the underlying field $k$ is of characteristic $p>0$ (by suitably ``factoring out'', when applicable, the inseparable part of the morphisms), we will assume, for the simplicity of exposition, that throughout the rest of the article $k$ is a complete, non-archimedean, nontrivially valued and algebraically closed field of characteristic 0. 

\subsection{$n$-radial morphisms of open discs}\hfill

\subsubsection{}\label{sec: disc and annuli}
To describe the structure of quasi-smooth $k$-analytic curves, we will use the notions of \emph{triangulations} as in \cite{Duc-book}, and before doing that, we will describe the basic pieces of $k$-analtic curves, namely, open (closed) discs and annuli.

Let $T$ be a coordinate on the Berkovich affine line $\A^1_k$ (that is, we identify $\A^1_k$ with the Berkovich spectrum $\cM(k[T])$). Then, for such  $T$ we will denote by

\begin{itemize}
\item $\td(a,r)$ (resp. $\td(a,r^{-})$) - closed (resp. open) Berkovich disc centered at $T=a\in k$ and of radius $r$. So, for example, $\td(a,r)=\{x\in \A^1_k\mid |T(x)-a|\leq r \}$. We will denote by $\eta^{_T}_{a,r}$ (or just $\eta_{a,r}$ if $T$ is clear from the context) the maximal point of the disc $\td(a,r)$. This is nothing but the multiplicative seminorm corresponding to the sup-norm on $\td(a,r)\cap k$;

\item $\ta(a;r_1,r_2)$, $a\in k$ and $r_1,r_2\in \R_{>0}$, $r_1<r_2$ - an open Berkovich annulus centered at $T=a$, of inner radius $r_1$ and outer radius $r_2$;

\item $\ta[a;r_1,r_2]$, $a\in k$ and $r_1,r_2\in \R_{>0}$, $r_1\leq r_2$ - a closed Berkovich annulus centered at $T=a$, of inner radius $r_1$ and of outer radius $r_2$.

\end{itemize}
If $a=0$ we will often drop writing it from the above notation and simply write, for example, $\td(r)$ instead of $\td(a,r)$. For $a\in k$, a coordinate given by $T\mapsto T-a$ will be denoted by $T_a$. For example, we have a canonical isomorphism of $A^{_T}(a;r_1,r_2)$ and $A^{_{T_a}}(r_1,r_2)$. Finally, if no confusion arises, we may drop writing $T$ in the notation above.\\
To these spaces we assign the corresponding rings of analytic functions. We denote by
\begin{itemize}
 \item $\tco(a,r^-)$ (resp. $\tco(a,r)$) - the ring of analytic functions on  $\td(a,r^-)$ (resp. on $\td(a,r)$) written as power series in $T-a$. For example, we have
\begin{align*}
\tco(a,r^-)&=\Big\{\sum_{i\geq 0}a_i(T-a)^i\mid \forall \rho\in (0,r)\quad \lim\limits_{i\rightarrow \infty}|a_i|\rho^i=0\Big\}\\
\tco(a,r)&=\Big\{\sum_{i\geq 0}a_i(T-a)^i\mid \lim\limits_{i\rightarrow \infty}|a_i|r^i=0\Big\};
\end{align*}
\item $\tco(a;r_1,r_2)$ the ring of analytic functions on $\ta(a;r_1,r_2)$;

\item $\tco[a;r_1,r_2]$ the ring of analytic functions on $\ta[a;r_1,r_2]$.
\end{itemize}

We say that a $k$-analytic curve $X$ is an open discs (resp. closed disc, resp. open annulus, resp. closed annulus) if it is isomorphic to an open disc $\td(a,r^-)$ for some $r>0$ and $a\in k$ (resp. $\td(a,r)$, resp. $\ta(a;r_1,r_2)$, resp. $\ta[a;r_1,r_2]$). If $X$ is an open disc and $T$ a coordinate on $X$ that identify it with an open annulus $\td(a,r^-)$ then to each point $x\in X$ we can assign the radius $r(x)$ (that depends on $T$) by setting $r(x)$ to be the infimum of radii of all closed discs in $\td(a,r^-)$ that contain $T(x)$ (and similarly we define the radius function if $X$ is a close disc or open or closed annulus). If $X$ is an annulus (open or closed), its \emph{skeleton} $\cS(X)$ is the set of points in $X$ that do not admit a neighborhood isomorphic to an open disc. For example, if $X=A(0;r_1,r_2)$, then $\cS(X)=\{\eta_{0,r}\mid r\in (r_1,r_2)\}$. We note that the restriction of the Berkovich topology on the skeleton of an annulus is the topology of the real segment.

\subsubsection{}
A morphism of open discs (and similarly for closed ones) $f:\td(r^-)\to \sd(r'^-)$ comes with a morphism of the corresponding algebras $f^\sharp:\sco(r'^-)\to \tco(r^-)$ which is generated by the image of $S$. In particular, $f^\sharp(S)$ is a power series in $\tco(r^-)$, hence of the form $\sum_{i\geq 0}f_i\cdot T^i$ where the coefficients $f_i\in k$ satisfy the usual convergence conditions above. Identifying $S$ with its image $f^\sharp(S)$ we may write $S=\sum_{i\geq 0}f_i\cdot T^i$ and call this the $(T,S)$-coordinate (or just coordinate) representation of the morphism $f$. Then, for a rational point $a\in \td(r^-)\cap k$, we have $f(a)=\sum_{i\geq 0}f_i\cdot a^i$. Furthermore, by changing $S$ to $S_{f(0)}$, we may (and most of the time will, unless otherwise stated) assume that $f_0=0$. Such a pair of coordinates $(T,S)$ will be called compatible. Note that for any $a\in \td(r)\cap k$, the pair of coordinates $(T_a, S_{f(a)})$ is compatible. We recall that for a morphism $f:\td(r^-)\to \sd(r'^-)$ to be finite of degree $n$ is equivalent to say that in $(T,S)$-representation of $f$ we have $|f_n|\cdot r^n=r'$ and $n$ is minimal with this property.

\begin{rmk}\label{rmk: general morphisms}
 Any nonconstant (\ie not in $k$) element in $\tco(0,r^-)$ of the form $f(T)=\sum_{i\geq 1}f_i\cdot T^i$ induces a surjective morphism $\sco(0,r^-)\to\tco(0,r'^-) $, where $S=f(T)$ and $r':=\sup\{|f_i|\cdot r^i\mid i\geq 1\}$. In general this morphism is quasi-finite (\cite[Section 3.1]{BerCoh}) and it is finite if and only if there exists an $i\geq 1$ such that $|f_i|\cdot r^i=r'$. The smallest such $i$ is then the degree of the morphism.
\end{rmk}

\subsubsection{}\label{ssec:valuation polygons} Suppose we are given an element $f(T)=\sum_{i\geq 0}f_i\cdot T^i\in \tco(r^-)$. Then we may study the valuation polygon $v(f(T),\cdot)$ of the the function  $\sum_{i\geq 1}f_i\cdot T^i$ which is a real function on the interval $(-\log r,\infty)$ given by 
$$
\rho\mapsto v(f(T),\rho):=\min\{-\log|f_i|+i\cdot\rho\mid i\geq 0\}=\min\{-\log|f(a)|\mid a\in k, -\log|a|=\rho\}.
$$
It is a standard fact that this function is a continuous, concave ($\cap$-shape) piece-wise affine with integral slopes. If $S=f(T)$ is a representation of a finite morphism of open discs, then the valuation polygon has finitely many break-points (since $f(T)$ has finitely many zeroes), hence finitely many slopes. The highest slope is then the degree of the morphism.

The multiplicative version of the valuation polygon is the following:

\begin{defn}
 Let $f:\td(r^-)\to \sd(r'^-)$ be a finite morphism of open discs. The function $\pf_{f,(T,S)}:[0,r]\to [0,r']$ defined by $\pf_{f,(T,S)}(0)=0$ and $\pf_{f,(T,S)}(r)=r'$ and for every $\rho\in (0,r)$ by $\pf_{f,(T,S)}=\max\{|f_i|\cdot\rho^i\mid i\in \N\}$ is called the \emph{$(T,S)$-profile} of $f$. 
 
 A point $\rho\in (0,r)$ is called a \emph{break-point} if $-\log \rho$ is a break-point of the valuation polygon of the $(T,S)$-representation of $f$. We will sometimes include $0$ and $r$ in the set of the break-points of $\pf_{f,(T,S)}$ and call these the trivial breaks.
 
 We call the slopes of the valuation polygon of $(T,S)$-representation of $f$ the \emph{dominating terms} and denote their set by $\Dc_{f,(T,S)}$.
\end{defn}

\begin{rmk} Strictly speaking, $\pf_{f,(T,S)}$ depends only on coordinate $T$ and not on its compatible pair $S$ (as long as $T$ and $S$ preserve the corresponding radii of the discs), but for informative reasons we keep both coordinates in notation.
\end{rmk}

\begin{rmk}\label{rmk: domin terms}
Let $f:\td(r^-)\to \sd(r'^-)$ be a finite morphism of open discs.

\emph{(1)} From the relation with the valuation polygon it follows that the $(T,S)$-profile of a finite morphisms of open discs is always a continuous and piecewise $|k|$-monomial function. Here is its more precise description that will be useful later on. Let $i_1<\dots<i_n$ be all the elements of $\Dc_{f,(T,S)}$ in the increasing order. Then, there exists real numbers $0=r_1<\dots<r_n$ such that the function $\pf_{f,(T,S)}$ is given by
$$
\pf_{f,(T,S)}=\begin{cases}
|f_{i_j}|\cdot\rho^{i_j},\quad \rho\in [r_{j},r_{j+1}], \quad j=0,\dots,n-1,\\
|f_{i_n}|\cdot\rho^{i_n},\quad \rho\in [r_n,r).
\end{cases}
$$
Moreover, for $j=1,\dots,n-1$ the equality $|f_{i_j}|\cdot r_{j+1}^{i_j}=|f_{i_{j+1}}|\cdot r_{j+1}^{i_{j+1}}$ holds.

Sometimes we will refer to the interval $[r_j,r_{j+1}]$ (or its interior) as the {\emph interval of dominance of $i_j$}. If there is no ambiguity, if we talk about valuation polygons, we will also call $(-\log r_j,-\log r_{j+1})$ the \emph{interval of dominance of $i_j$}.

\emph{(2)} It follows that $i\in \Dc_{f,(T,S)}$ if and only if there exists $\rho_1<\rho_2\in (0,r)$ such that  for every $\rho\in (\rho_1,\rho_2)$ and for every $j\neq i$ we have $|f_i|\cdot\rho^i>|f_j|\cdot\rho^j$.

\emph{(3)} For every $i\in \Dc_{f,(T,S)}$ and $j<i$ we have $|f_j|\cdot r^j<|f_i|\cdot r^i$. Indeed, by point \emph{(2)} above, for some $\rho\in (0,r)$ we have $|f_j|\cdot \rho ^j<|f_i|\cdot\rho ^i$, hence $|f_j|/|f_i|<\rho^{i-j}<r^{i-j}$.  

\emph{(4)} If for some $\rho_0\in [0,r)$, $|f_i|\cdot\rho_0^i>|f_j|\cdot\rho_0^j$ holds and $i<j$, then $|f_i|\cdot\rho^i>|f_j|\cdot\rho^j$ for all $\rho\in [0,\rho_0]$. This is because $|f_i|/|f_j|>\rho_0^{j-i}$ implies $|f_i|/|f_j|>\rho^{j-i}$ for all $\rho\in [0,\rho_0)$. On the other side, if $|f_i|\cdot\rho_0^i>|f_j|\cdot\rho_0^j$ and $i>j$, then $|f_i|\cdot\rho^i>|f_j|\cdot\rho^j$ for all $\rho\in [\rho_0,r]$. This is because $\rho^{i-j}>\rho_0^{i-j}>|f_j|/|f_i|$.
\end{rmk}
\subsubsection{}
Suppose that $f:\td(r^-)\to \sd(r'^-)$ is a finite morphism of open discs and $S=\sum_{i\geq 1}f_i\cdot T^i=\sum_{i\geq 1}f_i(0)\cdot T^i$ be its $(T,S)$-coordinate expression. For $a\in \td(r)\cap k$, we also have a $(T_a,S_{f(a)})$-coordinate representation of the form $S_{f(a)}=\sum_{i\geq 1}f_i(a)\cdot T^i_a$. 

\begin{defn}\label{def: coefficient}
 We call the function $f_{[i]}:\td(r)\to k$ given by $a\mapsto f_{[i]}(a):=f_i(a)$ the \emph{$i$-coordinate function} of $f$ (with respect to pair $(T,S)$).
\end{defn}

\begin{lemma}\label{lem: coordinate description} For each $i\geq 1$ and $a\in \td(r)\cap k$ we have:
 $$
 f_{[i]}(a)=\frac{1}{i!}\cdot\frac{d^iS}{dT^i}(a)=\sum _{j\geq 0}\binom{i+j}{j}\cdot a^j\cdot f_{i+j}.  
$$
In particular, there exists $b\in \td(r^-)\cap k$ such that for all $a\in k$ with $|b|=|a|$ we have $|f_{[i]}(b)|\geq |\binom{i+j}{j}||a|^j|f_{i+j}|$, for all $j\geq 0$. 

Furthermore, for every interval $(s_1,s_2)\subseteq(0,r)$, there exists a subinterval $(s',s'')$ such that for every $a\in \td(r^-)\cap k$, $|a|\in (s',s'')$, we have $|f_{[i]}(a)|\geq |\binom{i+j}{j}|\cdot|a|^j\cdot|f_{i+j}|$, for all $j\geq 0$.
\end{lemma}
\begin{proof}
The Taylor expansion of the function $S$ with respect to $T-a$ is 
$$
S=f(a)+\sum_{i\geq 1}\frac{1}{i!}\cdot\frac{d^iS}{dT^i}(a)\cdot(T-a)^i
$$
from which we obtain $S_{f(a)}=\sum_{i\geq 1}\frac{1}{i!}\cdot\frac{d^iS}{dT^i}(a)\cdot T_a^i$, that is $f_{[i]}(a)=\frac{1}{i!}\cdot\frac{d^iS}{dT^i}(a)$. On the other side we may also write, 
$$
S=\sum_{i\geq 1}f_i\cdot T^i=\sum_{i\geq 1}f_i\cdot(T_a+a)^i=f(a)+\sum_{i\geq 1}\bigg(\sum _{j\geq 0}\binom{i+j}{j}\cdot a^j\cdot f_{i+j}\bigg)\cdot T_a^i,
$$
which gives the second equality. The last two assertions follow from the definition of valuation polygon and from the fact that we can choose a subinterval $(s',s'')$ such that $f_{[i]}(T)$ does not have any zeroes of the norm belonging to $(s',s'')$. Then the valuation polygon of $f_{[i]}(T)$ has only one slope over $(-\log s'', -\log s')$ so that $|f_{[i]}(a)|=\max\limits_{i\geq 1}|\binom{i+j}{j}|\cdot|a|^j\cdot|f_{i+j}|$, for every $a$ such that $|a|\in (s',s'')$.
\end{proof}
It follows that coordinate functions extend to analytic functions on the whole $k$-analytic disc $\td(r)$, a fact that we will use freely in what follows.

\subsubsection{}

\begin{defn}\label{def: n-radial}
Let $f:\td(r^-)\to \sd(r'^-)$ be a morphism of open discs and let $n\in \N$.

 (1) We say that $f$ is an $n$-radial morphism if it is finite and there exists $s\in (0,r)$ such that for all $a\in \td(r)(k)$ the functions $\pf_{f,(T_a,S_{f(a)})}$ coincide on the interval $(s,r)$ and the corresponding valuation polygon has exactly $n$ slopes on the interval $(-\log r, -\log s)$. The infimum of such numbers $s$, which we will denote by $\b_{f,n}$ is called the border of $n$-radiality. By definition, we will say that $f$ is $0$-radial if it is finite. 
 
 (2) We say that $f$ is weakly $n$-radial if it is $(n-1)$-radial, and for every $a\in \td(r^-)(k)$, the $n$ first slopes of the $(T_a,S_{f(a)})$-valuation polygon of $f$ are the same.  
 
  If $f$ is a (weakly) $n$-radial morphism, the set of the first $n$ dominating terms of its valuation polygon (that is, the last $n$ dominating terms in its profile) expressed in some compatible coordinates is denoted by $\Dc_{f,n}$. 
\end{defn}
%
\begin{rmk}
 The notion of being (weakly) $n$-radial for a morphism $f$ does not depend on the chosen compatible coordinates $(T,S)$ (independent of whether they preserve the radii of the discs involved or not). However, if the coordinates do not preserve the radii of the discs, the border of $n$-radiality can change.
\end{rmk}

\begin{rmk}
 The relation between weakly $n$-radial and $n$-radial is suggested by the terminology. Namely, if $n\in \N$, every $n$-radial morphism is weakly $n$-radial by definition. Contrary does not hold, as shows the morphism $f:\td(1^-)\to \sd(1^-)$, given by 
 $$
 S=\alpha\cdot T+T^{2\cdot p},
 $$
 where we assume that $p\neq 2$ is the residual characteristic of $k$, and $|p|<|\alpha|<1$. It is not difficult to show that $f$ is weakly 1-radial (in fact, all finite morphisms of open discs are), while it is not 1-radial. See \cite[Remark 2.12]{Ba-Bo} for details.
\end{rmk}
%

\begin{rmk}\label{rmk: properties n-radial} We collect some of the more obvious properties of an $n$-radial morphism $f:D_1\to D_2$. For a more detailed study we refer to \cite{Ba-Bo}.
%

\emph{(1)} Let us denote the elements of $\Dc_{f,n}$ by $i_1<\dots <i_n$, and let $[r_j,r_{j+1}]$ be the interval of dominance of $i_j$, for $j=1,\dots,n$. Then, Remark \ref{rmk: domin terms} \emph{(1)} implies that for every $a\in \td(r^-)(k)$, the restriction of $\pf_{f,(T_a,S_{f(a)})}$ on the interval $[r_j,r_{j+1}]$ is given by $\rho\mapsto |f_{i_j}|\cdot \rho^{i_j}$.

\emph{(2)} Let $i\in \Dc_{f,n}$ and let $(r,r')$ be its interval of dominance. Then, for every $a\in \td(r^-)(k)$, every $\rho\in (r,r')$ and every $j\in \N$, $j\neq i$
$$
|f_{[j]}(a)|\cdot\rho^j\leq |f_i|\cdot\rho^i.
$$
Moreover, if $j<i$ we have the strict inequality. Explicitly, if $b\in \td(r^-)(k)$ with $|b|=|a|$ and such that $|f_{[j]}(b)|\geq \max |\binom{j+l}{l}|\cdot |a|^{j}\cdot |f_{j+l}|$ as in Lemma \ref{lem: coordinate description}, then it follows that for every $l\in\N$
$$
\left|\binom{j+l}{l}\right|\cdot|a|^j\cdot|f_{j+l}|\cdot\rho^j\leq |f_i|\cdot\rho^i,
$$
with strict inequality holding if $j<i$.
%

\emph{(3)} If $n\geq 2$ then $f$ is automatically $(n-1)$-radial, the boundary of $(n-1)$-radiality being the number $r$, where $(r,r')$ is the dominance interval of $\min\Dc_{f,n}$.

\emph{(4)} Suppose that $f:D_1\to D_2$ is a weakly $n$-radial morphism, $n\geq 2$ and $i\in\{1,\dots,n\}$. Let $b$ be the border of $i$-radiality of $f$ and let $D$ be an open disc in $D_1$ of radius $b$. Then, $f_{|D}$ is weakly $(n-i)$-radial. If $c\in (0,1)$ and $c>b$ while $c$ is smaller than the border of $(i-1)$-radiality (taken to be 1 if $i=1$), then the restriction of $f$ over a disc of radius $c$ is weakly $(n-i+1)$-radial.
\end{rmk}

The main arithmetic properties of $n$-radial morphisms are given in the following

\begin{lemma}\label{lem: n-radial properties}
 Let $f:\td(r^-)\to \sd(r'^-)$ be an $n$-radial morphism of open discs and let $i_1<\dots<i_n$ be its dominating terms. Then, 
 \begin{enumerate}
 \item For $l=1,\dots,n$ the coordinate function $f_{[i_l]}(T)$ is invertible on $\td(r)$;
  \item Assume $\text{char}(\wtilde{k})=p>0$. For each $l\in\{0,\dots,n\}$, $i_l$ is a power of $p$ and $p^{l}\,|\, i_l$. In particular, $i_n=\deg(f)$ is a power of $p$;
  \item Assume $\text{char}(\wtilde{k})=p>0$. Let $i\in \Dc_{f,n}$. 
  \begin{enumerate} 
  \item Suppose for some $j\neq i$, $|f_j|\cdot r^j\geq |f_i|\cdot r^i$. Then, $j>i$ and $i\,|\, j$;  
  \item Suppose for some $j\neq i$, $|f_j|\cdot r^j> |f_i|\cdot r^i$. Then $|j/i|<1$ and $$
  |\frac{j}{i     }|\cdot |f_j|\cdot r^j=|\binom{j}{i}|\cdot |f_j|\cdot r^j\leq|f_i|\cdot r^i.
  $$
  \end{enumerate}  
 \end{enumerate}
\end{lemma}
\begin{rmk}\label{rmk: kummer} We will use  Kummer theorem on valuation of binomial coefficients: If $\alpha>\beta$ are natural numbers and $a\in \N$ is the highest natural number such that $p^a$ divides $\binom{\alpha}{\beta}$, then $a$ is equal to the number of "carry-overs" when adding $\beta$ and $\alpha-\beta$ in the basis $p$. For example,
\begin{enumerate} 
\item All the numbers of the form $\binom{p^\alpha}{\beta}$, where $0<\beta< p^{\alpha}$, are divisible by $p$ and if $\alpha$ is not a power of $p$, there is always a $\beta\in \{1,\dots,p^\alpha-1\}$ such that $\binom{\alpha}{\beta}$ is not divisible by $p$. 
\item We have the equality $\left|\binom{a}{b}\right|=\left|\binom{a\cdot p^\alpha}{b\cdot p^\alpha}\right|$, where $a,b, \alpha\in \N$.
\end{enumerate}
\end{rmk}

\begin{proof}  
Let $\b_{f,n}$ be the border of $n$-radiality of $f$ and let us fix an $l\in \{0,\dots,n\}$.

{\emph{(1)}} It is enough to prove that $f_{[i]}(T)$ has no zeros on $\td(r)$, and for this it is enough to prove that $|f_{[i]}(a)|$ is constant for every $a\in \td(r)\cap k$. But the latter follows from the definitions of $n$-radiality and the remark that follows it. 

\emph{(2)} Suppose that $i_l$ is not a power of $p$, let $(s_1,s_1')\subset (0,r)$ be its interval of dominance, and $\rho\in (0,1)$ such that $\rho\cdot r\in (s_1,s_1')$. By Kummer theorem stated above there is a natural number $m$ such that $\binom{i_l}{m}$ is not divisible by $p$, hence has norm 1. Lemma \ref{lem: coordinate description} then implies that for $a\in \td(r^-)\cap k$ and $|a|$ close enough to $r$, the coefficient $f_{[m]}(a)$ will have the norm arbitrarily close to $r^{i_l-m}\cdot|f_{i_l}|$, so that in particular, since $m<i_l$, $|f_{[m]}(a)|\cdot\rho^m>r^{i_l-m}\cdot|f_{i_l}|\cdot\rho^{i_l}$. Then, $|f_{[m]}(a)|\cdot(r\cdot\rho)^m>|f_{i_l}|\cdot(r\cdot\rho)^{i_l}$, which is a contradiction. 
%
%

Since $p^0\,|\, i_0$, then if $p^j\,|\,i_j$, $i_{j+1}>i_j$ and $i_{j+1}$ being a power of $p$ imply that $p^{j+1}\,|\, i_{j+1}$.

\emph{(3)} It is clear that we may assume that $j>i=p^{\alpha_l}$ for both assertions. Let $(s,s')$ be the interval of dominance of $i$ and $\rho\in (0,1)$ such that $\rho\cdot r\in (s,s')$. \emph{(a)}. Suppose that $i\nmid j$ and let $j_0\equiv j \mod i$. Clearly $j_0<i$ and by Kummer's theorem $|\binom{j}{j_0}|=1$. Then, by Lemma \ref{lem: coordinate description} (see also part \emph{(2)} in the proof) there is an $a\in \td(r^-)\cap k$ such that $|f_{[j_0]}(a)|$ is arbitrarily close to $|f_{j}|\cdot r^{j-j_0}$, hence there is an $a\in \td(r^-)\cap k$ such that  $|f_{[j_0]}(a)|>|f_{j}|\cdot r^{j-j_0}\cdot \rho ^{i-j_0}\geq |f_i|\cdot(r\cdot \rho)^{i-j_0}$ so that $|f_{[j_0]}(a)|\cdot(r\rho )^{j_0}>|f_i|\cdot(r \rho )^i$, which is a contradiction.
\emph{(b)} By \emph{(a)} we have that $i\mid j$. Suppose that $|j/i|\cdot|f_j|\cdot r^j>|f_i|\cdot r^i$ and let $j=\sum_{s\geq s_0}n_s\cdot p^s$, $n_{s_0}\neq 0$ and $n_s\in\{0,\dots,p-1\}$, be the $p$-adic expansion of $j$. Then, $j/i=\sum_{s\geq s_0}n_s\cdot p^{n_s-\alpha_l}$ and $|j/i|=p^{-(s_0-\alpha_l)}$. On the other side, the $p$-adic expansion of $j-i$ is $\sum_{s>s_0}n_s\cdot  p^s+(n_{s_0}-1)p^{s_0}+\sum_{s_0>s\geq \alpha_l}(p-1)\cdot p^s$ and by Kummer's theorem we have $|\binom{j}{i}|=p^{-(s_0-\alpha_l)}$, as well. Then, as before, for $a\in \td(r^-)\cap k$ and $|a|$ close enough to $r$, we have by Lemma \ref{lem: coordinate description} that $|f_{[i]}(a)|\geq |\binom{j}{i}|\cdot|a|^{j-i}|\cdot f_{j}|>|f_i|\cdot r^{i-i}=|f_i|$ which contradicts \emph{(1)}. 
\end{proof}

\begin{cor}\label{cor: main all ineq}
 Let $f:\td(r^-)\to \sd(r'^-)$ be an $n$-radial morphism and $i\in \Dc_{f,n}$. Then, for every $m\in \N$ we have
 $$
 \left|\binom{m\cdot i}{i}\right|\cdot |f_{m\cdot i}|\cdot r^{m\cdot i}=|m|\cdot |f_{m\cdot i}|\cdot r^{m\cdot i}\leq |f_i|\cdot r^i
 $$
\end{cor}
\begin{proof}
 If $|f_{m\cdot i}|\cdot r^{m\cdot i}>|f_i|\cdot r^i$, then the statement is covered by the previous lemma. If $|f_{m\cdot i}|\cdot r^{m\cdot i}\leq|f_i|\cdot r^i$, then the claim is obvious.
\end{proof}

\subsubsection{}
The following class of morphisms was introduced and studied by M. Temkin in \cite{TeMu}.
\begin{defn}
 A finite morphism of open discs $f:\td(r)\to \sd(r')$ is said to be \emph{radial} if there exists $n\geq1$ such that $f$ is $n$-radial with $\b_{f,n}=0$. We will say that $n$ is characteristic (of radiality) of $f$. 
 
 If $f$ is radial then the function $\pf_f:=\pf_{f,(T,S)}$ does not depend on the couple of compatible coordinates $(T,S)$ (preserving the radii of discs) and will be called the \emph{profile} of $f$.
 \end{defn}
\begin{defn}\label{def: simple}
 We call the radial morphisms of characteristic 2, \emph{simple} morphisms.
\end{defn}

We may note that radial morphisms are in particular \'etale. As a direct consequence of Lemma \ref{lem: n-radial properties} we have the following corollary.

\begin{cor}\label{cor: radial properties}
 Assume $\text{char}(\wtilde{k})=p>0$. If $f$ is a radial morphism of open discs, then $\pf_f$ has local degrees powers of $p$. Moreover, if $p^\alpha$ and $p^\beta$ are two consecutive dominating terms, then $\alpha<\beta$.
 In particular, if $f:\td(1^-)\to \sd(1^-)$ is a radial morphisms of open unit discs, then $\pf_f\in\Lambda_p$.
\end{cor}

The proof of the following lemma amounts to a straightforward calculation which we omit.

\begin{lemma}\label{lem: profile change}
 Let $f:D^T(1^-)\to D^S(1^-)$ be a radial morphism of open unit discs with break-points $0=b_0<b_1<\dots<b_n<1=b_{n+1}$, local coefficients $a_1<\dots<a_{n+1}$ and local degrees $1=p^{\alpha_0}<p^{\alpha_1}<\dots<p^{\alpha_n}$. Let further $a\in k^{\circ\circ}$ and $r\in |k|\cap (0,1)$. Then, the restriction of $f$ to $D:=D^T(a,r^-)$ is again radial morphism. Moreover, the profile of $f_{|D}$ with respect to any pair of compatible coordinates identifying $D$
 and its image with open unit discs, is given by
 $$
 \rho\mapsto\pf_{f_{|D}}(\rho)=\frac{1}{a_i\cdot r^{p^{\alpha_{i-1}}}}\cdot\pf_{f}(r\cdot \rho),\quad \rho\in[0,1]
 $$
 where $i=1,\dots,n+1$ is such that $b_{i-1}\leq r< b_i$.
 \end{lemma}

\begin{rmk}
 From the valuation polygon point of view, radial morphisms which are even simpler than the simple ones are radial of characteristic 1, \ie isomorphisms and, as we know, in general they are far from being simple.
\end{rmk}

\begin{exa}
 We provide some examples of simple morphisms (assuming $\text{char}(\wtilde{k})=p>0$).
 
 \emph{(1)} Every \'etale morphism $f:\td(r^-)\to \sd(r'^-)$ of degree $p$ is simple. To see this, let $S=\sum_{i\geq 1} f_i\cdot T^i$ be the $(T,S)$-representation of $f$. Now, $f$ being \'etale means that the derivative $\sum_{i\geq 1} i\cdot f_i\cdot T^{i-1}$ is an invertible function on $\td(r^-)$, hence for every $i>1$ we have 
 \begin{equation}\label{eq: etale}
 |f_1|\cdot r\geq|i|\cdot |f_i|\cdot r^i.
 \end{equation}
 An immediate consequence is that $\pf_{f,(T,S)}$ has only one break. Namely, if $r_0:=\big(\frac{|f_1|}{|f_p|}\big)^{\frac{1}{p-1}}$, then the interval of dominance of $1$ is $[0,r_0]$. This is because for a small $\rho>0$, $|f_1|\cdot\rho>|f_i|\cdot \rho^i$ for every $i>1$ ($p$-adic Rolle theorem). Further, for $1<i<p$ inequality \eqref{eq: etale} becomes $|f_1|\cdot r\geq |f_i|\cdot r^i$, hence for such an $i$, $|f_1|\cdot \rho> |f_i|\cdot \rho^i$, for every $\rho\in [0,r)$. Since $p=\deg(f)$, and $|f_p|\cdot r_0^p=|f_1|\cdot r_0$ and $p$ dominates over some interval $(r-\eps,r)$, it follows that $(r_0,r)$ is the interval of dominance of $p$.
  
 If we take $a\in \td(r^-)\cap k$, and take a pair of compatible coordinates $(T_a,S_{f(a)})$, we note that $|f_{[1]}(a)|=|f_1|$ and $|f_{[p]}(a)|=|f_p|$ (first equality holds because $f$ is \'etale while the second because $p=\deg(f)$). Then, the same discussion above applies to deduce that the only break-point of $\pf_{f,(T_a,S_{f(a)})}$ is $\big(\frac{|f_{[1]}(a)|}{|f_{[p]}(a)|}\big)^{\frac{1}{p-1}}=r_0$. 
 
 \emph{(2)} Let $f:\td(1^-)\to \sd(1^-)$ be given by $S=T^{p^\alpha}+f_1\cdot T$, where $1>|f_1|>|p|$, $\alpha\in \N$. Then $f$ is simple of degree $p^\alpha$. Clearly, it is an \'etale morphism and $\pf_{f,(T,S)}$ has the only break-point in $r_0:=|f_1|^{\frac{1}{p^\alpha-1}}$. Now for $a\in \td(1^-)(k)$, and $l\in\{2,\dots,p^\alpha-1\}$, $|f_{[l]}(a)|=|\binom{p^\alpha}{l}|\cdot |a|^{p^\alpha-l}<|f_{[1]}(a)|=|f_1|$ by Remark \ref{rmk: kummer}, so $[0,r_0]$ is the interval of dominance of $1$ while $[r_0,1)$ is the interval of dominance of $p^\alpha$.
 
 \emph{(3)} Generalizing the previous example, let $f:\td(r^-)\to \sd(r'^-)$ be given by $S=f_{p^\alpha}\cdot T^{p^\alpha}+f_1\cdot T$. Let $r_0:=\big(\frac{|f_1|}{|f_{p^\alpha}|}\big)^{\frac{1}{p^\alpha-1}}$. Suppose that for every $l\in\{2,\dots,p^\alpha-1\}$ we have $\big|\binom{p^\alpha}{l}\big|\cdot r^{p^\alpha-l}\leq r_0^{p^\alpha-l}$ (a condition satisfied by the morphism in the previous example). Then, $f$ is simple. To see this, we note that $r_0$ is the break-point of $\pf_{f,(T,S)}$. If $a\in \td(r^-)\cap k$ then
 $$
 |f_{[l]}(a)|\cdot r_0^l=|\binom{p^\alpha}{l}|\cdot |a|^{p^\alpha-l}\cdot |f_{p^\alpha}|\cdot r_0^r<|\binom{p^\alpha}{l}|\cdot r^{p^\alpha-l}\cdot |f_{p^\alpha}|\cdot r_0^r\leq r_0^{p^\alpha-l}\cdot |f_{p^\alpha}|\cdot r_0^l=|f_1|\cdot r_0.
 $$
 The claim follows by similar arguments as before.
\end{exa}

\subsection{$n$-radial morphisms of quasi-smooth $k$-analytic curves}
From now on we assume that $\text{char}(\wtilde{k})=p>0$.
\subsubsection{}
In the present article we will deal with quasi-smooth $k$-analytic curves in the sense of Berkovich analytic geometry over $k$. We recall that a curve $X$ is quasi-smooth if every rational point in it has a neighborhood isomorphic to an open disc. We will use freely Berkovich classification of points of $X$ into four types, according to the nature of their completed residue field. More precisely, if $x\in X$ is a point, $\kappa(x)$ its residue field and $\sH(x)$ its completion with respect to the seminorm $|\cdot|_x$ induced by $x$ (after all, the points of $X$ are seminorms on $k$-affinoid algebras corresponding to their affinoid neighborhoods and we will use ``$|\cdot|_x$'' to denote the seminorm corresponding to $x$), then, fields $\sH(x)$ satisfy the Abhyankar inequality 
\begin{equation}\label{eq: Abhy}
\rk_\Q(|\sH(x)^*|/|k^*|\otimes_\Z\Q)+\trdeg(\wtilde{\sH(x)}/\wtilde{k})\leq 1, 
\end{equation}
where $\trdeg$ stands for ``transcendental degree''. Then, we have the following classification of points in $X$ (see \cite[Section 1.4.4.]{Ber90} and \cite[Section 3.6]{BerCoh} for the details): 1)\emph{of type 1} or \emph{rational} if $\sH(x)=k$; 2) \emph{of type 2} if $\trdeg(\wtilde{\sH(x)}/\wtilde{k})=1$. In this case field $\sH(x)$ is either isomorphic to the completion of $k(T)$ with respect to some norm coming from a closed disc in $k$ of a radius $r\in |k^*|$, either it will be a finite extension of such a field; 3) \emph{of type 3} if $\rk_\Q(|\sH(x)^*|/|k^*|\otimes_\Z\Q)=1$.  If $x$ is of type 3, then it admits a neighborhood isomorphic to a closed annulus $\ta[0;r_1,r_2]$ where $x$ corresponds to a point $\eta_r^T$, $r\notin|k|$. In this case, $\sH(x)=k_r$, where $k_r$ is the field of formal power series $\sum_{i\in \Z}a_i\, T^i$, $a_i\in k$ and $\lim\limits_{|i|\to \infty}|a_i|r^i=0$. 4) \emph{of type 4} if it is not of any of the previous three types. If $x$ is of type 4, we recall that from the existence of skeleta of $X$, $x$ admits a neighborhood isomorphic to an open unit disc, say $\td$. By Berkovich classification theorem, the point $x$ corresponds to a nested sequence of closed discs in $k$, $(D_n)_n$, with empty intersection and $\sH(x)$ is the completion of the the field $k(T)$ with respect to the family of norms $(D_n)_n$. We note that points of type 4 appear only if the base field $k$ is not maximally (spherically) complete. 
%

\subsubsection{}\label{sec: triangulations}
Let $X$ be a quasi-smooth $k$-analytic curve. We recall that a locally finite subset $\cT\subset X$ consisting of type 2 and 3 points is called a \emph{triangulation} of $X$ if $X\setminus\cT$ is a disjoint union of open \emph{unit} discs and annuli (we agree to consider an open disc punctured in one rational point to be an open annulus as well). Theorem ZZZ in \lc implies that quasi-smooth $k$-analytic curves admit triangulations. Moreover, for any locally finite subset $\cT'\subset X$ of type 2 points, there exists a triangulation $\cT$ of $X$ that contains $\cT'$. If $\cT$ is a triangulation of $X$, and if $\cA$ is the set of all the connected components of $X\setminus\cT$ that are open annuli, we denote by $\Gamma_{\cT}$ the \emph{skeleton of $X$ with respect to $\cT$}, which is the closed subset $\cT\cup\cup_{A\in \cA}\cS(A)$. In general, by a \emph{skeleton} $\Gamma$ of $X$ we will mean any subset of $X$ which is of the form $\Gamma_\cT$, for some triangulation $\cT$ of $X$. We note that $X\setminus\Gamma$ is a disjoint union of open unit discs and that a skeleton  $\Gamma$ can be empty only if $X$ is an open unit disc. For a $k$-analytic curve we say that it is \emph{basic} if it admits a triangulation consisting of one point.

If $\Gamma$ is a non-empty skeleton of a curve $X$, then there is a well defined deformation retraction $r_\Gamma:X\to \Gamma$ which is identity on $\Gamma$ and for a point $x\in X\setminus\Gamma$, $r_\Gamma(x)$ is the unique point on $\Gamma$ such that if $D$ is the connected component of $X\setminus\Gamma$ that contains $x$, then $r_{\Gamma}(x)$ is contained in the closure of $D$ in $X$.

\subsubsection{}\label{sec: tangent space}
If $x\in X$ of type 2, then the field $\wtilde{\sH}(x)$ has transcendence degree 1 over $\wtilde{k}$, hence is a function field of a smooth projective $\wtilde{k}$-algebraic curve, which we will denote by $\fC_x$. If $x$ is an interior point of $X$ then there is a 1-1 correspondence between the closed points of $\fC_x$ and the set $T_xX:=\varprojlim_{x\in U}U\setminus\{x\}$, where $U$ goes through the set of open neighborhoods of $x$ in $X$, as is detailed in \cite[Section 4.2.11.1]{Duc-book}. We will call the latter set the \emph{tangent space of $x$ in $X$} and use freely its identification with $\fC_x(\wtilde{k})$. 

If $x\in X$ is not in the interior of $X$, then the tangent space $\varprojlim_{x\in U}U\setminus\{x\}$ can only be identified with closed points of a $\wtilde{k}$-affine curve $\fC_x'$ whose smooth compactification is $\fC_x$.


\begin{defn}
 Let $X$ be a quasi-smooth $k$-analytic curve and $x\in X$ of type 2. Then, we define the \emph{genus of $x$}, denoted by $g(x)$, to be the genus of the curve $\fC_x$.
\end{defn}
%

%

\subsubsection{}
We next recall the category of germs of $k$-analytic spaces, referring to \cite[Section 3.4.]{BerCoh} for the details. Let $k$-Germs be the category whose objects are paris $(X,S)$, where $X$ is a $k$-analytic space and $S$ is a subset of the underlying topological space of $X$. A morphism between two such objects $(Y,T)\to (X,S)$ is given by a morphism $f:Y\to X$ such that $f(S)\subset T$. This category admits calculus of right fractions with respect to class of morphisms $f:(Y,T)\to (X,S)$, where $f$ induces an isomorphism between $Y$ and a neighborhood of $S$ in $X$. The resulting category is denoted by $k$-\emph{Germs}, and a morphisms $f:(Y,T)\to (X,S)$ in $k$-\emph{Germs} is given by an inductive limit of morphisms $f_{U}:U\to X$ where $U$ goes through the fundamental system of neighborhoods of $T$ in $Y$ and $f_U(T)\subset S$. Any such a morphism is called a representative of $f$. For example, two germs $(Y,T)$ and $(X,S)$ are isomorphic if there exists an open neighborhood $U$ of $T$ in $Y$ and $V$ of $S$ in $X$ and an isomorphism $f:U\to V$ inducing a bijection between $T$ and $S$. If $S$ consists of one point $\{x\}$ we write $(X,x)$ instead of $(X,\{x\})$.

\subsubsection{} Let $f:Y\to X$ be a finite \'etale morphism of quasi-smooth $k$-analytic curves. Then, it follows from simultaneous semistable reduction theorem that there exists a skeleton $\Gamma=(\Gamma_Y,\Gamma_X)$ of the morphisms where $\Gamma_Y$ and $\Gamma_X$ are skeleta of $Y$ and $X$, respectively, and $f^{-1}(\Gamma_X)=\Gamma_Y$. Following \cite{TeMu} we say that $\Gamma$ is \emph{radializing for $f$}, or that $f$ \emph{is radial with respect to $\Gamma$} if for every type 2 point $y\in \Gamma_Y$, every open disc $D$ in $Y\setminus \Gamma_Y$ which is attached to $y$, and any compatible coordinates $T$ and $S$ on $D$ and $f(D)$, respectively, which identify them with unit discs, the profile $\pf_{f_{|D},(T,S)}=:\pf_{f,y}$ only depends on $y$. As we have seen in Corollary \ref{cor: radial properties}, if $\text{char}(\wtilde{k})=p>0$, then $\pf_{f,y}\in\Lambda_p$. If we denote by $\Gamma_Y^{(2)}$ the set of type two points on $\Gamma_Y$, then we have a function $\pf_{f,\cdot}:\Gamma_Y^{(2)}\to \Lambda_p$ (or $\pf_{f,(\cdot)}(\cdot):\Gamma_Y^{(2)}\times [0,1]\to [0,1]$), $y\mapsto \pf_{f,y}$ (resp. $(y,r)\mapsto \pf_{f,y}(r)$). We sum up its main properties and the results of \cite{TeMu} in Theorem \ref{thm: Temkin's thm}.

Recall that a continuous function $f$ on a real interval (open or closed) $I$ is called monomial if there exists $a\in \R$ and $d\in\Z$ such that $f$ is of the form $r\mapsto a\,r^d$, for $r\in I$. We say that it is piece-wise monomial if we can cover $I$ by real intervals $(I_n)_n$ such that $f$ is monomial on every $I_n$. A break-point of $f$ is a point in the interior of $I$ where it is not smooth. It is not difficult to see that if $f$ has finitely many break-points then we can cover $I$ by finitely many intervals over which $f$ is monomial.

\begin{thm}[M. Temkin]\label{thm: Temkin's thm} Let $f:Y\to X$ be a finite \'etale morphism of quasi-smooth $k$-analytic curves.
 \begin{enumerate}
  \item Let $\Gamma'=(\Gamma_Y',\Gamma_X')$ be a skeleton of $f$. Then, there exists a radializing skeleton $\Gamma=(\Gamma_Y,\Gamma_X)$ of $f$ such that $\Gamma_Y'\subset \Gamma_Y$ and $\Gamma_X'\subset \Gamma_X$.
  \item The function $\pf_{f,\cdot}$ extends to a function (denoted in the same way) $\pf_{f,\cdot}:Y^{\text{hyp}}\to \Lambda_p$, where $Y^{\text{hyp}}$ is the set of type \emph{(2)}, \emph{(3)} and \emph{(4)} points, which varies continuously along skeleta of $Y$. More precisely, if $(y_n)_n$ is a sequence of points in $Y^{\text{hyp}}$ converging to some $y\in Y^{\text{hyp}}$, then $\pf_{f,y_n}$ converges uniformly to $\pf_{f,y}$.
  \item The function $\pf_{f,\cdot}$ is piece-wise $|k|$-monomial in the following sense: For every open annulus $A\iso \ta(0,r_1,1)$ in $Y$ which is not precompact, any $a\in [0,1]$, the function $r\mapsto \pf_{f,\eta^T_{0,r}}(a)$, where $r\in (r_1,1)$ is piece-wise monomial with finitely many break-points.
  \item Let $y\in Y^{\text{hyp}}$ and $x=f(y)$. The field $\sH(x)$ is almost monogeneous (see Section \ref{sec: almost monogeneous}) and $\H_{\sH(y)/\sH(x)}=\pf_{f,y}$.
 \end{enumerate}
\end{thm}
\begin{proof}
 For \emph{(1), (2)} and \emph{(3)} see \cite[Section 3.]{TeMu} and in particular Theorems 3.4.11. and 3.5.2. Part \emph{(4)} follows from Example 4.3.5. and Theorem 4.5.2. of \lc.  For another perspective on definition of $\pf_f$ to points of type \emph{(3)} and \emph{(4)} see also \cite{BojPoi}.
\end{proof}
The following lemma will be useful later on.
\begin{lemma}\label{lem: skeleta 2 of 3}
  Let $f:Y\to X$ be a finite morphism of quasi-smooth $k$-analytic curves with skeleton $\Gamma=(\Gamma_Y,\Gamma_X)$. Suppose there exists a quasi-smooth $k$-analytic curve $Z$ and finite morphisms $h:Y\to Z$ and $g:Z\to X$ such that $f=g\circ h$. Then, $\Gamma_Z:=g^{-1}(\Gamma_X)=h(\Gamma_Y)$ is a skeleton of $Z$.
 \end{lemma}
\begin{proof}
This is standard so we only sketch a proof. Since all the morphisms involved are finite, the ``branching'' points of $\Gamma_Z$ will be locally finite and it is enough to prove that each connected component of $Z\setminus \Gamma_Z$ is an open (unit) disc. For each such a component $U$ there is an open unit disc $D$ in $Y$ attached to $\Gamma_Y$ and an open unit disc $D'$ in $X$ attached to $\Gamma_X$ such that $f_{|D}:D\to D'$ is finite and factorizes through $h_{|D}:D\to U$ and $g_{|U}:U\to D'$.

Let $\eta_n$ be a sequence of type \emph{(2)} points in $D$ that converge to $\Gamma_X$ and let $D_n$ be unique closed disc in $D$ with Gauss point $\eta_n$. Then, the restriction of $f$ to $D_n$ is again a finite morphism of closed discs and consequently the image $U_n$ of $D_n$ by $h$ is $k$-affinoid domain in $U$ with good reduction (its Shilov point is $h(\eta_n)$). Passing to the canonical reduction, we obtain that $\wtilde{f_{|D_n}}:\A^1_{\wtilde{k}}\to\A^1_{\wtilde{k}}$ factors through a smooth $\wtilde{k}$-algebraic curve $\wtilde{U_n}$. This factorization extends to the extension of morphisms of smooth compactifications of our curves, that is, we have induced factorization $\P^1_{\wtilde{k}}\to \wtilde{U_n}'\to \P^1_{\wtilde{k}}$, where $\wtilde{U_n}'$ is a smooth compactification of $\wtilde{U_n}$. It follows that $\wtilde{U_n}'=\P^1_{\wtilde{k}}$ and coming back to analytic setting, it follows that $U_n$ is a $k$-affinoid domain with good reduction in $\P^1_k$. But then it is easy to see that it is isomorphic to a closed disc. Hence $U$ itself is contained in $\P^1_k$ (being an increasing union of closed discs) and since it is an image of an open disc by a finite map, it is necessarily an open disc. 
\end{proof}

\subsubsection{}
Let $f:Y\to X$ be a finite morphism of quasi-smooth $k$-analytic curves and let $y\in Y$ be of type 2, and $x=f(x)$. Then, morphism $f$ induces a finite morphism $\wtilde{f}_y: \fC_y\to \fC_x$. We say that $f$ is \emph{residually separable} (resp. \emph{residually purely inseparable}, resp. \emph{inseparable}) at $y$ if $\wtilde{f}_y$ is so. These notions extend to the points of type 3 and 4 as well, by a suitable scalar extension as is done in \cite[Definitions 1.17. and 1.18.]{BojPoi}. We denote the separability degree (resp. purely inseparable degree) of the induced field extension $\wtilde{k}(\fC_y)/\wtilde{k}(\fC_x)$ by $\fs_{f,y}$ (resp. $\fri_{f,y}$). We have (for $y$ and $x$ of type 2)
\begin{equation}\label{eq: sep insep}
|\sH(y):\sH(x)|=|\wtilde{\sH(y)}:\wtilde{\sH(x)}|=\fs_{f,y}\cdot\fri_{f,y}.
\end{equation}

\begin{lemma}\label{lem: sep insep}
Let $f:Y\to X$ be a finite morphism of quasi-smooth $k$-analytic curves, $y\in Y^{hyp}$ and $x=f(y)$. Then, if $\sH(y)/\sH(x)$ is simply ramified (resp. has trivial Herbrand function), then $f$ is residually purely inseparable (resp. is residually separable) at $y$. 

 A finite nontrivial extension $\sH(y)/\sH(x)$ is residually purely inseparable (resp. residually separable) if and only if the degree of $\H_{\sH(y)/\sH(x)}$ (degree of $\pf_{f,y}$) is equal to $\deg(f,y)=|\sH(y):\sH(x)|$ (resp. if $\H_{\sH(y)/\sH(x)}$ is trivial).
\end{lemma}
\begin{proof}
  We may assume that $y$ (hence $x$) are of type 2. The inseparability degree of $\wtilde{f}_y:\fC_y\to \fC_x$ is equal to the multiplicity of $\wtilde{f}_y$ at all but at most finitely many points in $\fC_y(\wtilde{k})$. Then, (by \cite[Th\'eor\`eme 4.3.13]{Duc-book}) this multiplicity is the degree of $f_{|D}$ for all but at most finitely many open discs $D$ in $Y$ attached to $y$ and this is the degree of $\pf_{f,y}=\H_{\sH(y)/\sH(x)}$. The claim follows. 
 \end{proof}
\begin{lemma}\label{lem: res sep at border}
 Let $f:D_1\to D_2$ be a finite, weakly $2$-radial morphism of open unit discs, let $p^\alpha,p^\beta\in \Dc_{f,2}$, $p^\beta<p^\alpha$ and let $b$ be the border of 1-radiality of $f$. Then, for every $a\in D_1(k)$ we have $\fs_{f,\eta_{a,b}}=p^{\alpha-\beta}$.
\end{lemma}
\begin{proof}
 Similarly as in the proof of the previous lemma, for all discs $D$ in $D_1$ that are attached to the point $\eta_{a,b}$ we have that $f_{|D}$ is of degree $p^\beta$. Indeed, because of our choice of $\eta_{a,b}$, $p^\beta$ is the highest slope of the valuation polygon of $f_{|D}$ (expressed in some compatible coordinates on $D$ and $f(D)$). Consequently, $p^\beta=\fri_{f,\eta_{a,b}}$. On the other side we have $\deg(f_{|D})=p^\alpha$, which one can check by counting the number of zeros of $f$ that are in $D$ (this time by choosing some compatible coordinates on $D_1$ and $D_2$). The result follows from \eqref{eq: sep insep}. 
\end{proof}
\begin{rmk}\label{rmk: insep degree}
 We will often use the following observations.
 
 \emph{(1)} Keeping the notation as before, suppose that $f$ is residually purely inseparable at $y$. Then, for every $\vt\in TyY$ and small enough annulus $A$ which is in $\vt$ and such that $f_{|A}$ is a finite morphism of open annuli, the degree of $f_{|A}$ is precisely $\deg(f,y)=\fri_{f,y}$. Indeed, this degree is equal to the multiplicity of the corresponding point in $\fC_y(\wtilde{k})$ but $\wtilde{f}_y$ being purely inseparable the multiplicity is the same for all the points in $\fC_y(\wtilde{k})$ and is equal to $\deg(\wtilde{f}_y)$. By Lemma \ref{lem: sep insep} this degree is also equal to degree of $\pf_{f,y}$.
 
 \emph{(2)} Continuing the previous point, if $x=f(y)$, then $g(x)=g(y)$. This is because the genus of the curve does not change under finite purely inseparable morphisms so we will have that $g(\fC_y)=g(\fC_x)$ (see \cite[Chapter IV, Proposition 2.5.]{Har77})
\end{rmk}

\subsubsection{}

If $f:Y\to X$ is a finite morphism of quasi-smooth $k$-analytic curves and $y\in Y$ and $x=f(y)$, then $f$ induces a finite extension $\sH(y)/\sH(x)$ of degree equal to local degree of $f$ at $y$ which we denote by $\deg(f,y)$ (this is geometric ramification index of $f$ at $y$ in terminology of \cite[Section 6.3.]{BerCoh}). By \cite[Remarks 6.3.1.]{BerCoh} there exists an open neighborhood $U$ of $y$ such that $f_{|U}:U\to f(U)$ is finite and $\deg(f_{|U})=|\sH(y):\sH(x)|$. Then, morphism $f_{|U}$ induces a morphism of germs $(Y,y)\to (X,x)$. This relation is given more precise in the Berkovich theorem \cite[Theorem 3.4.1.]{BerCoh} whose special instance, adapted to our interests, we next recall.
\begin{thm}[V. Berkovich]\label{thm: Berk thm}
Given a point $x\in X$, where $X$ is a quasi-smooth $k$-analytic curve, there is an equivalence of categories $F\'et(\sH(x))$ and $F\'et(X,x)$, where the former is the category of finite extensions of $\sH(x)$ and latter is the category of finite morphisms of $k$-germs $(Y,S)\to (X,x)$.
\end{thm}

\begin{rmk}
 In the \lc the situation is more general. The base field $k$ is of arbitrary characteristic, $x$ is a point in a Berkovich space $X$, and the categories involved are the same except one considers finite separable extensions of $\sH(x)$ and morphisms of germs $(Y,S)\to (X,x)$ which have a finite \'etale representative. Since we assumed that $k$ is of characteristic 0 every finite extension of $\sH(x)$ is separable and every finite morphisms of germs will have an \'etale representative.
\end{rmk}

\section{Factorization of morphisms of quasi-smooth $k$-analytic curves}
We assume until the rest of the article that the base field $k$ in addition satisfies $\text{char}(\wtilde{k})=p>0$.
\subsection{Local factorizations}

\subsubsection{}

 Let $f:Y\to X$ be a finite morphisms of quasi-smooth $k$-analytic curves and let $U\subset Y$ be a connected $k$-analytic curve. 
 \begin{defn}
 We say that $\fU=(U_i,f_i)_{i=1,\dots,m}$ is a \emph{factorization of $f$ over $U$} if $U_1=U$, and if the morphisms $f_i:U_i\to U_{i+1}$ , $i=1,\dots,m-1$ and $f_{m}:U_{m}\to f(U)$ are finite morphism of quasi-smooth $k$-analytic curves. We say that $\fU$ is of \emph{length $m$}. 
 
 If $y\in Y$ and $\fU$ as above, we say that $\fU$ is a factorization of $f$ at $y$ if $U_1$ is a neighborhood of $y$ in $Y$. 

Given two factorizations $\fU=(U_i,f_i)$ and $\fU'=(U'_i,f'_i)$ of $f$ over some $k$-analytic curves $U$ and $U'$ in $Y$, of length $m$, we say that $\fU$ is a \emph{subfactorization} of $\fU'$, or that \emph{$\fU$ is induced by $\fU'$}, and write $\fU\subset \fU'$, if $U_i\subset U'_i$ and $f_i=f'_{|U_i}$, for each $i=1,\dots,m$.

Finally, let $\fU=(U_i,f_i)$ and $\fU'=(U'_i,f'_i)$ be two factorizations of $f$ over $U$ of the same length $m$. We say that they are \emph{isomorphic} if there exists isomorphisms $g_i:U_i\to U'_i$ for $i=1,\dots,m$, such that for any such $i$, $f'_i\circ g_i=g_{i+1}\circ f_i$. 
\end{defn}

\subsubsection{}
As we have seen in Theorem \ref{thm: Temkin's thm} the fields $\sH(x)$ are almost monogeneous so in particular Theorem \ref{thm: canonical towers} applies. In combination with Theorem \ref{thm: Berk thm} we obtain the following.

\begin{thm}\label{thm: local fact}
 Let $f:Y\to X$ be a finite morphism of quasi-smooth $k$-analytic curves and let $y\in Y^{\text{hyp}}$ and $x=f(y)$. Suppose that $\pf_{f,y}$ has characteristic $n$. Then, there exists a neighborhood $U$ of $y$ in $Y$, a factorization $\fU=(U_i,f_i)_{i=1,\dots,n}$, of $f$ over $U$ such that
 \begin{enumerate}
  \item Let us put $y_1:=y$ and $y_{i}=f_{i-1}(y_{i-1})$, for $i=2,,\dots,n$. Then,
  \begin{enumerate}
   \item Each $f_i$, $i=1,\dots,n$, is residually purely inseparable with simple profile at $y_i$.
   \item The morphism $f_{n+1}$ is residually separable and in fact an isomorphism if $f$ is residually purely inseparable at $y$.
  \end{enumerate}
  \item The factorization $\pf_{f,y}=\pf_{f_n,y_n}\circ\dots\circ\pf_{f_1,y_1}$ is the canonical factorization of $\pf_{f,y}$ in $\Lambda_p$.
  \item Given any other factorization $\fU'$ of $f$ at $y$ satisfying the properties \emph{1} and \emph{(2)} above, then there are subfactorizations of $\fU$ and $\fU'$ at $y$ which are canonically isomorphic.
 \end{enumerate}

\end{thm}

\begin{proof}
 Let 
 $$
 \sH(y)=L_1/L_2/\dots/L_n/L_{n+1}=K'/\sH(x)
 $$
 be the canonical tower for the extension $\sH(y)/\sH(x)$, as in Theorem \ref{thm: canonical towers}. By Berkovich theorem \ref{thm: Berk thm}, for each $L_i$, $i=1,\dots,n+1$, there exists a quasi-smooth, $k$-analytic curve, say $Y_i$, and a point $y_i\in Y_i$ (with $Y_1$ a neighborhood of $y$ in $Y$, $y_1=y$), such that $\sH(y_i)\simeq L_i$. The same theorem implies that there exist finite morphisms $f_i:(Y_i,y_i)\to (Y_{i+1},y_{i+1})$, $i=1,\dots, n$, (resp. $f_{n+1}:(Y_{n+1},y_{n+1})\to (V',x)$, where $V'$ is a neighborhood of $x$ in $X$) that induce
extensions $\sH(y_i)/\sH(y_{i+1})$ (resp. $\sH(y_{n+1})/\sH(x)$). 

We next choose a compatible set of representatives of morphisms $f_i$. To this end, let $f_i:Y_{i,0}\to Y_{i+1}'$ be a representative of $f_i$ such that $Y_{i+1}'\subset Y_{i+1,0}$, for $i=1,\dots,n$ (note that we can choose such representatives inductively), and let $V$ be an open neighborhood of $x$ in $X$ that is contained in $f_{n+1}\circ\dots\circ f_1(Y_{1,0})$. Finally, let $U_i:=f^{-1}_i\circ\dots\circ f_{n}^{-1}(V)$. Clearly, $U_i$ is an open neighborhood of $y_i$ in $Y_i$ and $f_i:U_i\to U_{i+1}$ is a finite morphism that induces the extension $\sH(y_i)/\sH(y_i)$. Also, each $f_i$, $i=1,\dots,n$, is residually purely inseparable at $y_i$ while $f_n$ is residually separable by Lemma \ref{lem: sep insep}. The claims \emph{(1)} and \emph{(2)} now follow by invoking Theorem \ref{thm: Temkin's thm}.

For the uniqueness, suppose $\fU'=(U'_i,f'_i)_{i=1,\dots,n+1}$ is another factorization of $f$ at $y$ satisfying the conditions \emph{(1)} and \emph{(2)} ($\fU'$ has to have the same length as $\fU$ because of the conditions!). We note that the morphism $f_i':(U_i',y_i')\to (U_{i+1}',y_{i+1}')$ induces the extension $\sH(y_i)/\sH(y_{i+1})$ isomorphic to $L_i/L_{i+1}$ (because of properties \emph{(1)} and \emph{(2)}), hence isomorphic to $\sH(y_i)/\sH(y_{i+1})$. By Theorem \ref{thm: Berk thm} there are isomorphisms of $k$-\textit{Germs} $g_i:(U_i,y_i)\to (U_i',y_i')$ and $g:(V,x)\to (V',x')$, such that the following diagrams are commutative for $i=1,\dots,n$, 
$$
 \begin{tikzcd}
  (U_i,y_i)\arrow{r}{g_{i}} \arrow[swap]{d}{f_{i|U_i}} & (U'_i,y'_i) \arrow{d}{f'_{i|U'_i}}\\
  (U_{i+1},y_{i+1}) \arrow{r}{g_{i+1}}& (U'_{i+1},y_{i+1}'),
 \end{tikzcd}
 \quad\text{and}\quad
 \begin{tikzcd}
  (U_{n+1},y_{n+1})\arrow{r}{g_{n+1}} \arrow[swap]{d}{f_{n+1|U_{n+1}}} & (U'_i,y_i) \arrow{d}{f'_{n+1|U'_{n+1}}}\\
  (V,x) \arrow{r}{g}& (V',x').
 \end{tikzcd}
$$

By choosing compatible representatives similarly as before, the theorem follows.
\end{proof}
\begin{rmk}
  By shrinking the curves $U_i$ constructed above, if necessary, we can make sure that they are all basic curves (see Section \ref{sec: triangulations}). 
\end{rmk}

\begin{defn}
 Let $f:Y\to X$ be a finite morphism of quasi-smooth $k$-analytic curves and let $y\in Y^\text{hyp}$. Let $U_i$ and $f_i$ be as in Theorem \ref{thm: local fact}. We call the data $\fU_{f,y}:=(U_i,f_i)_{i=1,\dots,n}$ the \emph{canonical factorization of $f$ at $y$}.
 
 We say that $\fU_y$ is canonical factorization by \emph{basic curves} if all the curves in $\fU_{f,y}$ are basic. 
 
 
 If $f$ is clear from the context we may omit writing it in the index.
\end{defn}

The following lemma suggests that we can glue ''compatible`` factorizations.

\begin{lemma}\label{lem: gluing factorizations}
 Let $f:Y\to X$ be a finite morphism of quasi-smooth $k$-analytic curves, let $\fU_1$ and $\fU_2$ be factorizations of $f$ over some $k$-analytic subsets $U_1$ and $U_2$ of $Y$, respectively. Suppose there is a nonempty $k$-analytic subsets $U_0\subset U_1\cap U_2$ such that the restrictions $\fU_{1|U_0}$ and $\fU_{2|U_0}$ are isomorphic. Then, there exists a factorization of $f$ over $U_1\cup U_2$, unique up to a canonical isomorphism, whose restrictions to $U_1$ and $U_2$ are isomorphic to $\fU_1$ and $\fU_2$, respectively.
\end{lemma}
\begin{proof}
 We note that the two factorizations have the same length as they have isomorphic subfactorizations. Let us write $\fU_i=(U_{i,j},f_{i,j})_{j=1,\dots,n}$ and $\fU_{i|U_0}=(U_{i,j,0},f_{i,j,0})_{j=1,\dots,n}$ for $i=1,2$, and let $g_j:U_{1,j,0}\to U_{2,j,0}$, $j=1,\dots,n$, be the isomorphism inducing the isomorphism of the factorizations. Then, by \cite[Proposition 1.3.3.]{BerCoh} there is a $k$-analytic curve $U'_j$, unique up to a canonical isomorphism, and obtained by ``gluing'' the curves $U_{1,j}$ and $U_{2,j}$ via $g_j$. Consequently, the morphisms $f_{1,j}$ and $f_{2,j}$ glue as well and if we denote the corresponding morphism by $f_j$, then the factorization $\fU:=(U'_j,f_{j})_{j=1,\dots,n}$ is the one we were looking for.
\end{proof}

%
%

\subsection{Global factorizations}\label{sec: global}\hfill

Continuing the discussion of factorizations we note that given two factorizations $\fU$ and $\fU'$ of finite morphisms of quasi-smooth $k$-analytic curves $f:Y\to X$ and $g:X\to Z$, respectively, one can concatenate them in a natural way in order to obtain a factorization of $g\circ f: Y\to Z$. The following definition will allow us to compactly state the global factorization results that follow.

\begin{defn}\label{front-refined}
 Let $f:Y\to X$ be a finite morphism of quasi-smooth $k$-analytic curves and let $\fU=(U_i,f_i)_{i=1,\dots,n}$ and $\fU'=(U'_i,f'_i)_{i=1,\dots,m}$, $m>n$, be two factorizations of $f$ (over $Y$). We say that $\fU$ can be {\em front-refined} to $\fU'$ if there exists a factorization of $\fU''$ of $f_1:Y=U_1\to U_2$ such that $(U_i,f_i)_{i=2,\dots,n}$ concatenated to $\fU''$ is isomorphic to $\fU'$.
 
 Let $y\in Y$. We say that $\fU$ can be {\em locally front-refined} at $y$ to $\fU'$ if there exists a subfactorization of $\fU$ at $y$ that can be front-refined to a subfactorization of $\fU'$ at $y$. 
\end{defn}

Finally, to be able to extend local canonical factorizations to global ones, we  need a stronger notion of radiality.
 \begin{defn}\label{def: uniform n radial}
  Let $f:Y\to X$ be a finite morphism of quasi-smooth $k$-analytic curves and let $\Gamma=(\Gamma_Y,\Gamma_X)$ be a skeleton of $f$.  
  
 We say that $f$ is \emph{uniformly weakly $n$-radial} (resp. \emph{$n$-radial}, resp. \emph{radial}) \emph{with respect to $\Gamma$} if $\fs_{f,y}$ is the same for all $y\in\Gamma_Y$ and if $f_{|D}$ is weakly $n$-radial (resp. $n$-radial, resp. radial) for every connected component  $D$ in $Y\setminus \Gamma_Y$ with $\Dc_{f_{|D},n}$ (resp. $\Dc_{f_{|D}}$) not depending on $D$.
 
  We say that $f$ is \emph{uniformly residually separable with respect to $\Gamma$} if $f$ is weakly $1$-radial with respect to $\Gamma$ and is residually separable at every point $y\in\Gamma_Y$ of constant degree along $\Gamma_Y$.
 \end{defn}

 \begin{rmk}
  If the skeleton $\Gamma=\emptyset$, then the previous notion of uniformly (weakly) $n$-radial morphism with respect to $\Gamma$ reduces to the notion of (weakly) $n$-radial morphism of open unit discs.
 \end{rmk}


\begin{thm}\label{thm: global fact}
  Let $f:Y\to X$ be a finite morphism of quasi-smooth $k$-analytic curves, and let $\Gamma=(\Gamma_Y,\Gamma_X)$ be a skeleton of $f$. Suppose that $f$ is uniformly weakly $n$-radial with respect to $\Gamma$. Then, there exists a factorization of $f$ of the form
  \begin{equation}\label{eq: global can fact}
  Y\xrightarrow[]{g}Z\xrightarrow{h}Z'\xrightarrow{h'}X
  \end{equation}
   such that
   \begin{enumerate}
    \item The morphism $h'$ is uniformly residually separable with respect to $(h\circ g(\Gamma_Y),\Gamma_X)$ and of degree equal to $\fs_{f,y}$ (for any $y\in\Gamma_Y$).
    \item The morphism $g$ is uniformly weakly $(n-1)$-radial with respect to $(\Gamma_Y,g(\Gamma_Y))$, while the morphism $h$ is uniformly radial of characteristic 1 with respect to $(g(\Gamma_Y),h\circ g(\Gamma_Y))$.
    \item Moreover, for every point $y\in\Gamma_Y$, the factorization \eqref{eq: global can fact} can be locally front refined at $y$ to a canonical factorization of $f$ at $y$. 
   \end{enumerate}
   The factorization \eqref{eq: global can fact} is unique up to a canonical isomorphism.
 \end{thm}


\begin{rmk}\label{rmk: global fact for discs}
 In case that $\Gamma=\emptyset$, that is when $f$ is a finite morphism of open unit discs, we note that the theorem simply asserts the factorizaton of $f$ into a simple morphism (that is $h$) and an $(n-1)$-radial morphism (that is $g$).  The ``canonicity'' of the factorization comes from the fact that for every point $y\in Y$ that is the Shilov point of a disc in $Y$ of radius close enough to 1, the factorization can be locally front refined at $y$ to a canonical factorization of $f$ at $y$. 
\end{rmk}

  We start with some preliminary results. In particular, in the next lemma we will first show that we can factor out the ``residually separable'' part of the morphism $f$, that is we will show the existence of $Z'$, $h\circ g$ and $h'$, and then later we will deal with existence of $g$ and $h$. Although the arguments at places will be almost repetitive (after all, the main tool will be Theorem \ref{thm: local fact} together with Lemma \ref{lem: gluing factorizations}), we hope that separating the discussion of $h'$ from $h\circ g$ will make the presentation more clear.
 
 \begin{lemma}\label{lem: sep fact}
 Keep the setting as in Theorem \ref{thm: global fact}. Then, there exists a factorization of $f$ of the form 
 \begin{equation}\label{eq: fact help}
 Y\xrightarrow[]{H}Z'\xrightarrow{h'}X,
 \end{equation}
 such that 
 \begin{enumerate}
  \item The morphism $h'$ is uniformly residually separable with respect to $(g(\Gamma_Y),\Gamma_X)$ and of degree $\fs_{f,y}$ for any $y\in \Gamma_Y$ while the morphism $H$ is uniformly weakly $n$-radial with respect to $(\Gamma_Y,g(\Gamma_Y))$. Moreover, for every point $y\in \Gamma_Y$, $H$ is residually purely inseparable at $y$.
  \item For any $y\in \Gamma_Y$, the factorization \eqref{eq: fact help} can be locally front refined at $y$ to a canonical factorization of $f$ at $y$.
 \end{enumerate}
The curve $Z'$ and morphisms $H$ and $h'$ are unique up to a canonical isomorphism.
\end{lemma}
 
\begin{proof}
 It follows from Theorem \ref{thm: local fact} that for each $y\in\Gamma_y$ there exists a factorization $\fU_y$ of $f$ at $y$ (locally at $y$ unique up to a canonical isomorphism) of the form 
 \begin{equation}\label{eq: help fact}
 U_{y}\xrightarrow[]{f_{y,i}} U'_{y}\xrightarrow[]{f_{y,s}} f(U_{y}),
 \end{equation} 
 where $U_y$ is a basic curve and a connected open neighborhood of $y$ in $Y$, $U'_y$ is a basic curve and $f_{y,i}$ and $f_{y,s}$ are finite morphisms such that $f_{y,i}$ is residually purely inseparable at $y$ of degree $\fri_{f,y}$ while $f_{y,s}$ is residually separable at $f(y)$ of degree $\fs_{f,y}$. By shrinking $U_y$ if necessary, we may assume that $f_{y,i}$ remains residually purely inseparable of degree $\fri_{f,y}$ at every point $y'\in\Gamma_Y\cap U_y$, and this in turn implies that $f_{y,s}$ is residually separable at $f_{y,i}(y')$ necessarily of degree $\fs_{f,y}$. Furthermore, by shrinking some of such obtained factorizations (for various $y\in\Gamma_Y$) we may find a locally discrete subset $P\subset \Gamma_Y$ such that: 1) $\Gamma_Y\subset \cup_{y\in P}U_y$, 2) for every three distinct points $y,y',y''\in P$, $U_y, U_{y'}$ and $U_{y''}$ have an empty intersection and 3) for every two distinct $y,y'\in P$, $U_y\cap U_{y'}$ is either empty or it is contained in an open annulus in $Y$ whose skeleton is contained in $\Gamma_Y$. Now, if $y,y'\in P$ are such that $U_y\cap U_{y'}$ is not empty, by constructions of factorizations $\fU_y$ and $\fU_{y''}$, for every $y''\in\Gamma_Y\cap U_y\cap U_{y'}$, $\fU_y$ and $\fU_{y'}$ are locally isomorphic at $y''$ hence they glue according to Lemma \ref{lem: gluing factorizations}. By gluing all the factorizations $\fU_y$ for $y\in P$, we obtain quasi-smooth $k$-analytic curves $Y'$ and $Z''$, where $\Gamma_Y\subset Y'\subset Y$ and finite morphisms $H':Y'\to Z''$ and $h'':Z''\to X'$ with $f_{|Y'}=h''\circ H'$ and such that $H'$ is residually purely inseparable at the points of $\Gamma_Y$ while $h''$ is residually separable at the points of $H'(\Gamma_Y)$ and local isomorphism outside $H'(\Gamma_Y)$ (due to degree reasons). Furthermore, for every $y\in \Gamma_Y$, the obtain factorization by construction can be locally front refined at $y$ to a canonical factorization of $f$ at $y$.
 
 If $Y'=Y$ we are done because by construction $H'$ and $h''$ satisfy all the assertions of the lemma, so suppose that $Y'$ is properly contained in $Y$. By shrinking $Y'$ if necessary, we may assume that $Y\setminus Y'$ is a disjoint union of closed discs disjoint from $\Gamma_Y$ and such that if $D$ is a connected component of $Y\setminus Y'$, then so are open discs in $f^{-1}(f(D))$. Moreover, we may also assume that no connected component of $Y\setminus \Gamma_Y$ contains more than one of such discs. Let $D$ be a connected component (open disc) of $Y\setminus\Gamma_Y$ that is not contained in $Y'$ (so that $D$ contains a closed disc which is a connected component of $Y\setminus Y'$). In order to extend the constructed factorization also over $D$ and its $f$-conjugates, we note that $D\cap Y'$ is an open annulus, say $A$, $f_{|A}$ is a finite morphism with $f(A)$ an open annulus and so is $H'(A)$ assuming that we shrank $Y'$ enough. Furthermore, $\deg(h''_{|H'(A)})=1$  so $h''_{H'(A)}$ is an automorphism of $H'(A)$. But then it is easy to see that the two factorizations 
 $$
 f_{|D}:D\xrightarrow[]{f_{|D}}f(D)\xrightarrow[]{\id} f(D)\quad \text{and}\quad f_{|A}:A\xrightarrow[]{H'_{|A}} H'(A)\xrightarrow[]{h''_{|H'(A)}} f(A)
 $$
 glue via $h''^{-1}_{|H'(A)}$. By doing this for all the discs in $f^{-1}(f(D))$ and for all connected components of $Y\setminus Y'$, and by denoting the corresponding curve and morphisms by $Z'$, $H:Y\to Z'$ and $h':Z'\to X$, respectively, we obtain a global factorization of $f$, that satisfies properties \emph{(1)} and \emph{(2)} from the lemma.
 
 The asserted uniqueness of the curve $Z$ and morphisms $h$ and $g$ is also clear from the construction of the local factorizations and their uniqueness.
 \end{proof}

\begin{lemma}\label{lem: fact discs}
Theorem \ref{thm: global fact} is true when
\begin{enumerate}
 \item[Case 1.] $Y$ and $X$ are closed unit discs and $\Gamma=(\{\eta\},\{\zeta\})$, where $\eta$ and $\zeta$ are Shilov points of $Y$ and $X$, respectively.
 \item[Case 2.] $Y$ and $X$ are open unit discs and $\Gamma$ is empty (see Remark \ref{rmk: global fact for discs} for the statement in this case).
\end{enumerate}

\end{lemma}

\begin{proof}
 We simultaneously consider both cases and argue by induction on degree $\deg(f)$, the case $\deg(f)=1$ being the base of induction and trivial.

Suppose that the claim is true in both cases for all morphisms satisfying the conditions of the theorem of degree smaller than some $m$ and suppose that $\deg(f)=m$. In case 1. if $\fs_{f,\eta}>1$ then by Lemma \ref{lem: sep fact} we may ``factor out'' the separable part and apply induction on the remaining morphism of closed discs and we are done. So suppose in addition for the first case that $\fs_{f,\eta}=1$, that is, $f$ is residually purely inseparable at $\eta$. Let $p^{\alpha}<p^{\beta}$ be the elements of $\Dc_{f,2}$.

Let $b$ be the border of $1$-radiality of $f$ and let $D$ be a closed disc in $Y$ of radius $b$ and with Shilov point $\xi$. Then $\fs_{f,\xi}>1$, and $f_{|D}$ is weakly $1$-radial with respect to $(\{\xi\},\{f(\xi)\})$ (see Remark \ref{rmk: properties n-radial}) \emph{(4)} and Lemma \ref{lem: res sep at border}), so Lemma \ref{lem: sep fact} applies and we conclude that there exists a factorization of $f$ over $D$ of the form 
\begin{equation}\label{eq: help fact 1}
D\xrightarrow{f_{|D,i}} D_s\xrightarrow{f_{|D,s}} f(D),
\end{equation}
where $f_{|D,i}$ is uniformly weakly $(n-1)$-radial with respect to $(\xi,f_{|D,i}(\xi))$ and is residually purely inseparable at $\xi$ while $f_{|D,s}$ is residually purely separable at $f_{|D,i}(\xi)$ (we also keep in mind other properties of \eqref{eq: help fact 1} granted by Lemma \ref{lem: sep fact}). On the other side, it follows from the canonical factorization of $f$ at $\xi$ that there is an open neighborhood $U_{\xi}$ of $\xi$ in $Y$ and a factorization of $f$ over $U_{\xi}$
\begin{equation}\label{eq: help fact 2}
 U_{\xi}\xrightarrow[]{f_{\xi,i}} U'_{\xi}\xrightarrow[]{f_{\xi,s}} f(U_{\xi}),
 \end{equation}
with $f_{\xi,i}$ residually purely inseparable at $\xi$ and $f_{\xi,s}$ residually purely separable at $f_{\xi,i}(\xi)$, and moreover, the two factorizations \eqref{eq: help fact 1} and \eqref{eq: help fact 2} are isomorphic over some $k$-analytic subset of $Y$ containing $\xi$ (because they are both constructed from canonical factorization of $f$ at $\xi$). By Lemma \ref{lem: gluing factorizations} they glue and by shrinking $U_{\xi}$ if necessary, we may assume that there exists a factorization of $f$ over some open disc $D'_\xi$ that contains $D$ and of the form 
\begin{equation}\label{eq: help fact 3}
 D'_\xi\xrightarrow{f_{|D'_\xi,i}} D''_\xi\xrightarrow{f_{|D''_\xi,s}} f(D'_\xi),
\end{equation}
where $f_{|D'_\xi,i}$ is residually purely inseparable at $\xi$ and $f_{|D''_\xi,s}$ is residually purely separable at $f_{|D'_\xi,i}(\xi)$. We further note that by shrinking $D'$ if necessary, we can (and will) assume that the factorization \eqref{eq: help fact 3} at every point $\xi'\in D'_\xi$ which is of radius bigger or equal to $b$, can be locally front refined at $\xi'$ to a canonical factorization of $f$ at $\xi'$, or more precisely, $f_{|D''_\xi,s}$ will be the last factor in the canonical factorization of $f$ at $\xi'$. This is because: 1) it is a factor of $f$ (in some neighborhood of $\xi'$); 2) of the continuity of the profile $\pf_f$ at $\xi$ and its canonical factors (at the corresponding points) and 3) of the fact that $f_{|D'}$ is weakly $n$-radial, so that $\pf_{f_{|D''_{\xi},s},f_{|D'_{\xi},i}(\xi')}$ will be the last factor in $\pf_{f,\xi}$. 

Let now $\xi\in Y$ be of radius bigger than $b$ and let 
\begin{equation}\label{eq: help fact 4} 
 U_\xi\xrightarrow{f'_{\xi}} U'_\xi\xrightarrow{f''_{\xi}}f(U_\xi)
\end{equation}
be a factorization of $f$ over some simple neighborhood of $\xi$ so that for each point $\xi'\in \Gamma_{U_\xi}$ ($\Gamma_{U\xi}$ being the skeleton of $U_\xi$ coming from $\{\xi\}$), $f''_\xi$ is the last factor in the canonical factorization of $f$ at $\xi'$, or in other words, factorization \eqref{eq: help fact 4} can be locally front refined at $\xi'$ to a canonical factorization of $f$ at $\xi'$. By an argument similar to what we discussed before, we can always find such a neighborhood $U_\xi$ of $\xi$.

Finally, for each $\xi\in Y$ of radius bigger or equal to $b$, we choose a factorization $\fU_\xi$ of $f$ at $\xi$ of the form \eqref{eq: help fact 3} or \eqref{eq: help fact 4} according to whether $\xi$ is of radius $b$ or bigger than $b$, respectively. Furthermore, by removing and shrinking some of them, we may assume that: 1) the domains of factorizations cover $Y$, 2) domains of any 3 distinct factorizations have an empty intersection. If $\fU_{\xi_1}$ and $\fU_{\xi_2}$ are such two factorizations with intersecting domains, then it is easy to see by their constructions that they satisfy the conditions of Lemma \ref{lem: gluing factorizations}, and in particular they glue. By gluing all the factorizations, we obtain a factorization of $f$ over $Y$ of the form
\begin{equation}
 Y\xrightarrow{g} Z\xrightarrow{h} X.
\end{equation}
By (the proof of) Lemma \ref{lem: skeleta 2 of 3}, we know that $Z$ is a closed (resp. open) disc in the first (resp. second) case. We also know that, by constructions of the factorizations, for each $\xi$ of radius $r(\xi)\geq b$, $g$ is of degree $p^\beta$ at $\xi$ while $h$ is of degree $p^{\alpha-\beta}$ at $g(\xi)$ which is of radius $r(g(\xi))\geq b^{p^\beta}$. On the other side, if $\xi\in Y$ and $r(\xi)<b$, then $h$ is of degree 1 at $g(\xi)$. In conclusion, $h$ is a radial morphism of open discs with simple profile which in the first case coincides with the last factor in the canonical factorization of $\pf_{f,\eta}$, while $g$ is weakly $(n-1)$-radial which follows by comparing profiles of $h$ and $g\circ g$ with respect to some compatible coordinates on $Y$, $Z$ and $X$. The rest of the claims follows easily. 
\end{proof}

\begin{proof}[Proof of Theorem \ref{thm: global fact}]

We may assume that $\Gamma_Y\neq \emptyset$ as this case is covered in the previous lemma. Further, Lemma \ref{lem: sep fact} implies that we may assume that $f$ is uniformly residually purely inseparable with respect to $\Gamma$.

Similarly as in the proofs of lemmas \ref{lem: sep fact} and \ref{lem: fact discs} we start by choosing for each $y\in \Gamma_Y$ a factorization $\fU_{y}$ of $f$ at $y$ of the form
\begin{equation}\label{eq: main help}
 U_y\xrightarrow{g_y} U'_y\xrightarrow{h_y} f(Y_y), 
\end{equation}
where $U_y$ is a simple curve with the skeleton $\Gamma_{U_y}$, $h_y$ is the last factor in some canonical factorization of $f$ at $y$ (so that $g_y$ is the composition of the remaining factors).  Furthermore, by shrinking and deleting some of the $\fU_y$, we may assume that there exists a locally finite subset $P\subset\Gamma$ such that: 1) The union $\cup_{y\in P}U_y$ covers $\Gamma_Y$; 2) for three distinct $y,y',y''\in P$, $U_{y}\cap U_{y'}\cap U_{y''}=\emptyset$ and 3) for a $y\in P$ and $y'\in \Gamma_{U_y}$, the restriction of the factorization \eqref{eq: main help} (to some neighborhood of $y'$) can be refined to a canonical factorization of $f$ at $y'$ and (the suitable restriction of) $h$ is the last factor in it (we can impose this condition on $\fU_y$ by a similar argument as in the proof of Lemma \ref{lem: fact discs} just after equation \eqref{eq: help fact 3}). Consequently, we may glue factorizations $\fU_y$, $y\in P$ and in such a way we obtain a $k$-analytic curve $Y_0\subset Y$ that contains $\Gamma_Y$ and a factorization $\fU_0$ of $f$ over $Y_0$ which we denote by 
$$
Y_0\xrightarrow{g_0} Z_0\xrightarrow{h_0} X_0.
$$
We next proceed similarly as in the proof of Lemma \ref{lem: sep fact}. If $Y_0\neq Y$, then by shrinking $Y_0$ if necessary, we may assume that $Y\setminus Y_0$ is a disjoint union of closed discs, and if $D'$ is one such a disc contained in an open unit disc $D$ attached to a point $y\in\Gamma_Y$, then the radius of $D'$ (in $D$) is bigger than $b_y$, where $b_y$ is the border of $1$-radiality of $f_{|D}$ (which depends only on $y$ and not on a chosen disc $D$). In particular this condition implies that $A:=D\cap Y_0$ is an open annulus and $f^{-1}(f(A))=A$. Moreover, the skeleton of $A$ is part of the skeleton of $Y_0$, and the restrictions of factorizations $\fU_D$ from Lemma \ref{lem: fact discs} and $\fU_0$ over $A$ are isomorphic, so in particular they glue. Continuing this process for all the discs in $Y\setminus Y_0$ we obtain the factorization $\fU$ of $f$ over $Y$ of the form 
$$
Y\xrightarrow{g} Z\xrightarrow{h} X.
$$

That the given factorization satisfies the conditions of Theorem \ref{thm: global fact} follows from its local construction and Theorem \ref{thm: local fact}. The theorem is proved.
\end{proof}

\subsubsection{}

We end this section by spelling out an immediate corollary of the previous theorem, which concerns uniformly radial morphisms.
\begin{cor}\label{cor: global fact}
  Let $f:Y\to X$ be a finite morphism of quasi-smooth $k$-analytic curves, and let $\Gamma=(\Gamma_Y,\Gamma_X)$ be its skeleton. Suppose that
 $f$ is uniformly radial with respect to $\Gamma$ and of characteristic $n$.
Then, there exists a factorization $\fU=(Y_i,f_i)_{i=1,\dots,n+1}$ of $f$ over $Y$, unique up to a canonical isomorphism and such that for each $y\in \Gamma_Y$, $\fU$ induces a canonical factorization of $f$ at $y$.

\end{cor}

\section{Harmonicity properties of morphisms of Berkovich curves}

\subsection{Preliminary results}
\subsubsection{} Let $f:A\to A'$ be a finite morphism of open annuli, and let $T$ and $S$ be coordinates on $A$ and $A'$ identifying them with $A^T(0;r,1)$ and $A^S(0;r',1)$, respectively. Then, $f$ expressed in $(T,S)$ coordinates may be written as
\begin{equation}\label{eq: morph coord}
S=\sum_{i\in\Z}a_i\, T^i,
\end{equation}
and since it is a finite morphism of open annuli, the valuation polygon of the right-hand side of \eqref{eq: morph coord} has only one segment with slope $d\in\Z$. Then, the absolute value $|d|_{\infty}$ is the degree of the morphism. We say that $T$ and $S$ are \emph{aligned} if  $d=|d|_\infty=\deg(f)$. This is equivalent to say that the radius of points $f(\eta^T_{\rho})$ goes to 1, as $\rho \to 1$. 

Suppose that $T$ and $S$ are aligned. Since $f$ is \'etale, the derivative $\frac{dS}{dT}$ has no zeroes on $A^T(0;r,1)$ hence we can write it as 
 $$
\frac{dS}{dT}=\eps_{f,(T,S)}\cdot T^{f,\sigma_{(T,S)}}\cdot (1+h(T)),
$$
where $\eps_{(T,S)}\in k^\circ$, and for every $\rho\in (r,1)$, $|h(T)|_{\eta^T_{\rho}}<1$. 

The following is \cite[Lemma 4]{BojRH}.
\begin{lemma}\label{lem: sigma prop}
 If $T'$ and $S'$ are coordinates on $A$ and $A'$ aligned with $T$ and $S$, respectively, then $\sigma_{(T,S)}=\sigma(T',S')$. Furthermore, if we have a finite \'etale morphism $g:A^{S}(0;r',1)\to A^U(0;r'',1)$, where $S$ and $U$ are aligned then 
 $$
 \sigma_{g\circ f,(T,U)}=\deg(f)\cdot\sigma_{g,(S,U)}+\sigma_{f,(T,S)}.
 $$
\end{lemma}
We will write $\sigma(f)$ or $\sigma$ instead of $\sigma_{f,(T,S)}$ if the aligned coordinates $T$ and $S$ and $f$ are clear from the context.

If now $h$ is a real function defined on the interval $(r,1)$ we put
$$
\partial h:=\lim_{\rho\to 1}\frac{d\log h(\rho)}{d\log \rho},
$$
provided that the righthand side exists.

\begin{lemma}\label{lem: partial exists 1}
 Keep the setting $f:A\to A'$ and notation as above and assume that $T$ and $S$ are aligned. Suppose further that the profile $\pf_{f,\eta^t_{\rho}}$ has constant characteristic $n+1$, $n\geq0$, independent of $\rho$, for $\rho\in (r,1)$. Let $0<b_1(\rho)<\dots<b_n(\rho)<1$ be its break-points and $1=p^{\alpha_0}<p^{\alpha_1}<\dots<p^{\alpha_n}$ be its local degrees. Then 
 \begin{equation}\label{eq: sigma and breaks}
 |\eps(f)|\cdot \rho^{\sigma(f)-d+1}=\prod_{i=1}^{n}\big(b_{i}(\rho)\big)^{p^{\alpha_{i}}-p^{\alpha_{i-1}}},
 \end{equation}
 where in case $n=0$ we take the right-hand side to be 1. 
 Furthermore, $\partial b_i$ exists and
 \begin{equation}\label{eq: partial sigma and breaks}
 \sigma(f)=\sum_{i=1}^{n}\big(p^{\alpha_{i}}-p^{\alpha_{i-1}}\big)\cdot\big(\partial b_{i}-1\big)+d-p^{\alpha_{n}}.
 \end{equation}
\end{lemma}
\begin{proof}
It follows from \cite[Section 4.2.4.]{CTT14} that the left-hand side of \eqref{eq: sigma and breaks} is equal to the different of the field extension $\sH(\eta^T_\rho)/\sH(\eta^S_{\rho^d})$ (as defined in \lc). Lemma 4.2.2. in \lc implies that this different is 1 if and only if $f$ is residually separable at $\eta^T_\rho$, that is, if and only if $n=0$. In the other case, \cite[Corollary 4.4.8.]{TeMu} and Lemma \ref{lem: lambda prop} imply that this different is equal to the right-hand side of \eqref{eq: sigma and breaks}.

Our morphism $f$ satisfies the conditions of Theorem \ref{thm: global fact}, so there exists annuli $A^{T_i}(0;r_i,1)$, $i=1,\dots,n+1$ and morphisms $f_{i}:A^{T_i}(0;r_i,1)\to A^{T_{i+1}}(0;r_{i+1},1)$, for $i=1,\dots,n$ and $f_{n+1}:A^{T_{n+1}}(0;r_{n+1},1)\to A^S(0;r^d,1)$ (we disregard $f_{n+1}$ if $f$ is residually purely inseparable at each point of the skeleton of $A$) so that $\big(A^{T_i}(0;r_i,1),f_i\big)_{i=1,\dots,n}$ induces a canonical factorization of $f$ at each $\eta^T_{\rho}$, $\rho\in (r,1)$, and so that the pairs of coordinates $(T_i,T_{i+1})$, $i=1,\dots,n$ and $(T_{n+1},S)$ are aligned. Theorem \ref{thm: local fact} also implies that: 1) the degree of $f_i$ is $p^{\alpha_i-\alpha_{i-1}}$, 2) for a $j=1,\dots,n-1$ and $\rho\in (r,1)$ we have $f_j\circ\dots\circ f_1(\eta^T_\rho)=\eta^{T_{j+1}}_{\rho^{p^{\alpha_j}}}$ and 3) the break-point of $f_{j+1}$ at $\eta^{T_{j+1}}_{\rho^{\alpha_i}}$ is $b_{f_{i+1}}(\rho^{p^{\alpha_j}})=b_{i+1}^{p^{\alpha_i}}(\rho)$.

Equation \eqref{eq: sigma and breaks} applied to the morphism $f_i$ gives
$$
\eps(f_i)\cdot\big(\rho^{p^{\alpha_i}}\big)^{\sigma(f_i)-p^{\alpha_i-\alpha_{i-1}}+1}=b_{f_i}(\rho^{p^{\alpha_{i-1}}})^{p^{\alpha_i-\alpha_{i-1}}-1}=\big(b^{p^{\alpha_{i}}}_{i+1}(\rho)\big)^{p^{\alpha_i-\alpha_{i-1}}-1}=\big(b_{i}(\rho)\big)^{p^{\alpha_i}-p^{\alpha_{i-1}}}.
$$
Consequently, $\partial b_i$ exists (and is equal to $\frac{1}{p^{\alpha_i}-p^{\alpha_{i-1}}}\cdot(\sigma(f_i)-p^{\alpha_i-\alpha_{i-1}}+1)$).

Then, \eqref{eq: sigma and breaks} implies that 
$$
\sigma(f)-d+1=\sum_{i=1}^n(p^{\alpha_i}-p^{\alpha_{i-1}})\cdot\partial b_i
$$
which implies \eqref{eq: partial sigma and breaks} once we note that
$$
d-1=d-p^{\alpha_n}+\sum_{i=1}^n(p^{\alpha_i}-p^{\alpha_{i-1}}).
$$
\end{proof}

\subsection{Harmonicity properties}
\subsubsection{} 
Let $f:Y\to X$ be a quasi-finite morphism of quasi-smooth $k$-analytic curves. Let $y\in Y$ be a type 2 point in the interior of $Y$, and let $x=f(y)$. We denote by $\deg(f,y)$ be the local degree of $f$ at $y$. It follows from semistable reduction theorem that for a small enough neighborhood $U_y\in Y$, $U_y\setminus\{y\}$ is a disjoint union of open unit discs and (finitely many) open annuli, which are in correspondence with the tangent space $T_yY$ (see Section \ref{sec: tangent space}). Let us denote by $A_{\vt}:=A^{T_{\vt}}(0;r_{\vt},1)$ an open annulus contained in $U_y$ and corresponding to the tangent point $\vt$, and where the coordinate $T_\vt$ is chosen in such a way that the points $\eta_{\vt,\rho}:=\eta^{T_{\vt}}_\rho$ converge to $y$, when $\rho\to 1$ (here and elsewhere, $\eta_{\vt,\rho}$ is the unique point on the skeleton of $A_{\vt}$ of radius $\rho$).

For each $\vt\in T_yY$ and small enough annulus $A_\vt\in \vt$ the restriction $f$ to $A_\vt$ induces a finite \'etale morphism of open annuli, that is, there is an annulus $A_\vv\in T_xX$ such that $f_{\vt}:=f_{|A_\vt}:A_\vt\to A_\vv$ is finite \'etale. If we choose coordinate $T_\vt$ and $S_\vv$ on $A_\vt$ and $A_\vv$ identifying them with $A^{T_\vt}(0;r_\vt,1)$ and $A^{S_\vv}(0;r_\vv,1)$, respectively, and so that the pair $(T_\vt,S_\vv)$ is aligned and the points $\eta^{T_\vt}_\rho$ converge to $y$ as $\rho\to 1$ (hence, $\eta^{S_\vv}_\rho \to x$ as $\rho \to 1$), having in mind Lemma \ref{lem: sigma prop}, we put
$$
\sigma_{\vt}(f):=\sigma_{f,T_\vt,S_\vv)}\quad \text{and}\quad d_{\vt}:=\deg(f_\vt).
$$
Then, we have the following local Riemann-Hurwitz formula.
\begin{thm}(See \cite[Theorems 7 and 8 ]{BojRH}, \cite[Theorem 4.5.4.]{CTT14})\label{thm: RH}
 $$
 \deg(f_y)(2-2\cdot g(x))-(2-2\cdot g(y))=\sum_{\vt\in T_yY}\sigma_{\vt}(f).
 $$
\end{thm}

Let now $f_{\vt}:(r_{\vt},1)\to \R_{>0}$ be a function. We set 
$$
\partial_{\vt}f_{\vt}:=\lim_{\rho\to 1}\frac{d\log f_{\vt}(\rho)}{d\log \rho},
$$
provided that the right-hand side exists.

\subsubsection{}
We keep the setting $f:Y\to X$, notation and assumptions as before. Let $\pf_{f,y}$, as usual, denote the profile of $f$ at $y$ and suppose it is of characteristic $n+1$, $n\geq 0$. Let $0<b_1<\dots<b_n<1$ be the break-points of $\pf_{f,y}$. By continuity of the profile function, for each $\vt\in T_yY$ and small enough open annulus $A_{\vt}=A^{T_{\vt}}(0;r_{\vt},1)$ in $\vt$, the profile $\pf_{f,\eta_{\vt,\rho}}$, $\rho\in (r_{\vt},1)$ will have a fixed characteristic, say $n(\vt)+1$ and we have $n(\vt)\geq n$. Let $0<b_{\vt,1}(\rho)<\dots<b_{\vt,n(\vt)}(\rho)<1$ be the break-points of $\pf_{f,\eta_{\vt,\rho}}$, and $1=p^{\alpha_{\vt,0}}<\dots<p^{\alpha_{\vt,n(\vt)}}$ be the corresponding local degrees. Let now, for each $i=1,\dots,n$, $I_{\vt,i}$ be the set of those numbers $j$ for which $b_{\vt,j}(\rho)$ converges to $b_i$ as $\rho\to 1$ and let $I_{\vt,0}$ be the set of those numbers $j$ for which $b_{\vt,j}(\rho)$ converges to 1, as $\rho\to 1$.  We note that if $n=0$ these sets as well may be empty.

\begin{lemma}\label{lem: partial exists 2}
 For each $\vt\in T_yY$, $i=0,\dots,n$, such that $I_{\vt,i}\neq \emptyset$ and $j\in I_{\vt,i}$, $\partial_{\vt} b_{\vt,j}$ exists and for all but at most finitely many $\vt$, we have
 \begin{equation}\label{eq: partial b}
 \partial_{\vt} b_{\vt,j}-1=0.
 \end{equation}
\end{lemma}
\begin{proof}
 That $\partial_{\vt}b_{\vt,j}$ exists follows from Lemma \ref{lem: partial exists 1}. Let $\vt\in T_yY$ be such that if $D\in \vt$ is an open disc in $Y$ attahed to $y$, the restriction $f_{|D}$ is radial morphism of open (unit) discs. We note that this will be the case for all $\vt\in T_yY$ except for at most finitely many. Also, the profile of the morphism $f_{|D}$ coincide with $\pf_{f,y}$. 
 
 Let $A^{T_\vt}(0;r_{\vt},1)$ be an open annulus in $\vt$ so that $r_\vt>b_n$, and let $\rho\in (r_\vt,1)$. Then, Lemma \ref{lem: profile change} implies that the breakpoints of the profile of $f$ at $\eta^{T_\vt}_\rho$ are $b_{\vt,i}(\rho)=\rho^{-1}b_i$, $i=1,\dots,n$. The claim follows.
\end{proof}

The main results of this section are the following formulas which refine Riemann-Hurwitz formula \ref{thm: RH} and describe harmonicity properties of morphisms at type 2 points.

\begin{thm}\label{thm: loc harmonicity}
 Keep the notation as above and let $i\in\{0,1,\dots,n\}$. 
 \begin{enumerate}
\item If $i\neq 0$, then
$$
 \big(p^{\alpha_{i}}-p^{\alpha_{i-1}}\big)\cdot\big(2-2\cdot g(y)\big)=\sum_{\vt\in T_yY}\Big(\sum_{j\in I_{\vt,i}}(p^{\alpha_{\vt,j}}-p^{\alpha_{\vt,j-1}})\cdot\big(\partial_{\vt}b_{\vt,j}-1\big)\Big).
 $$
 \item If $i=0$, then
 $$
 \deg(f,y)\cdot\big(2-2\cdot g(x)\big)-p^{\alpha_n}\cdot\big(2-2\cdot g(y)\big)=\sum_{\vt\in T_yY}\Big(\sum_{j\in I_{\vt,0}}(p^{\alpha_{\vt,j}}-p^{\alpha_{\vt,j-1}})\cdot\big(\partial_{\vt}b_{\vt,j}-1\big)+d_{\vt}-p^{\alpha_{\vt,n(\vt)}}\Big).
 $$
 \end{enumerate}
\end{thm}

\begin{lemma}\label{lem: harmon simply ram}
 Theorem \ref{thm: loc harmonicity} is true if $\sH(y)/\sH(x)$ is simply ramified. 
\end{lemma}
\begin{proof}
 Indeed, in this case for each $\vt\in T_yY$ the set $I_{\vt,0}$ is empty and $I_{\vt,1}=\{1,\dots,n(\vt)\}$. Moreover, for each $\vt\in T_yY$, $d_{\vt}=p^{\alpha_{\vt,n(\vt)}}$. Since $f$ is residually purely inseparable at $y$ as well, then $g(y)=g(x)$. Applying Theorem \ref{thm: RH} and Lemma \ref{lem: partial exists 1} we obtain 
 $$
 p^{\alpha_{1}-\alpha_0}\cdot(2-2\cdot g(y))-(2-2\cdot g(y))=\sum_{\vt\in T_yY}\sum_{i=1}^{n(\vt)}(p^{\alpha_{\vt,i}}-p^{\alpha_{\vt,i-1}})\cdot(\partial_{\vt}b_{\vt,i}-1),
 $$
 which we were supposed to prove.
\end{proof}

\begin{lemma}\label{lem: harmon sep}
 Theorem \ref{thm: loc harmonicity} is true if $\sH(y)/\sH(x)$ is residually separable, that is $\pf_{f,y}=\id$. 
\end{lemma}
\begin{proof}
 In this case the local degree of $f$ at $y$ is $\fs_{f,y}$ and  $I_{\vt,0}$ is either empty or equal to $\{1,\dots,n(\vt)\}$. If it is empty, then $n(\vt)=0$ and the morphism is residually separable at every point $\eta_{\vt,\rho}$, $\rho\in (r_\vt,1)$. Hence the different of the extension $\sH(\eta_{\vt,\rho})/\sH(f(\eta_{\vt,\rho}))$ is 1 and \eqref{eq: sigma and breaks} and \eqref{eq: partial sigma and breaks} imply that $\sigma=d_\vt-1$. If $I_{\vt,0}=\{1,\dots,n(\vt)\}$ then by using Lemma \ref{lem: partial exists 1}, Theorem \ref{thm: RH} becomes
 $$
 \fs_{f,y}\cdot(2-2\cdot g(x))-(2-2\cdot g(y))=\sum_{\vt\in T_yY}\left(\sum_{i\in I_{\vt,0}}\big(p^{\alpha_{\vt,i}}-p^{\alpha_{\vt,i-1}}\big)\cdot\big(\partial_{\vt}b_{\vt,i}-1\big)+d_{\vt}-p^{\alpha_{\vt,n(\vt)}}\right),
 $$
which is true also in case $I_{\vt,0}=\emptyset$. The lemma follows.
 \end{proof}

\begin{proof}[Proof of Theorem \ref{thm: loc harmonicity}] 

Let us fix $\vt\in T_yY$ and let $\fU_y=(U_i,f_i)_{i=1,\dots,m}$ be a canonical factorization of $f$ over some neighborhood of $y$ in $Y$ and let us put $y_1=y$ and $y_{i+1}:=f_i(y_i)$ (we note that $m=n$ or $m=n+1$ depending on whether $f$ is residually purely inseparable or not at $y$). The restriction of $\fU_y$ over some small enough open annulus $A_\vt=A^{T_\vt}(0;r_{\vt},1)$ (oriented as usual) can be refined to a canonical factorization of $f$ over $A_{\vt}$. For a $\rho\in (r_{\vt},1)$ let $0<b_{\vt,1}(\rho)<\dots<b_{\vt,n(\vt)}(\rho)<1$ be all the break-points of $\pf_{f,\eta_{\vt,\rho}}$.  

We put $\vt_1:=\vt$ and inductively $\vt_{i+1}=f_i(\vt_i)$. Let us further fix an $i=1,\dots,m$ and consider the morphism $f_{\vt_i}$ which is the restriction of $f_i$ to a (small enough) open annulus $A_{\vt_i}=A^{T_{\vt_i}}(0;r_{\vt_i},1)$ in $\vt_i$. Then for $\rho\in (r_{\vt_i},1)$ let $0<b_{\vt_i,1}(\rho)<\dots<b_{\vt_i,n(\vt_i)}(\rho)<1$ be all the break-points of the profile $\pf_{f_i,\eta_{\vt_i,\rho}}$ and let $1=p^{\alpha_{\vt_i,0}}<\dots<p^{\alpha_{\vt_i,n(\vt_i)}}$ be the corresponding local degrees. We note that by our choice of $A_{\vt_i}$ (that is, since $A_{\vt_i}$ is chosen small enough), we may assume that $n(\vt_i)$ does not depend on $\rho\in (r_{\vt_i},1)$. 

By continuity of the profiles involved, the canonical factorization of each $f_{\vt_i}$ can be completed to a canonical factorization of $f_{\vt}$. Using this, by comparing the profiles  $\pf_{f_\vt}(\eta^{T_\vt}_\rho)$ and $\pf_{f_{\vt_1}}(\eta^{T_{\vt_1}}_\rho)$ we conclude that 
$$
n(\vt_1)=|I_{\vt,1}|
$$
and moreover 
$$
b_{\vt,l}(\rho)=b_{\vt_1,l}(\rho),\quad l=1,\dots,n(\vt_1),
$$
and the corresponding local degrees match. Consequently, 
$$
\partial_\vt b_{\vt,l}=\partial_{\vt_1} b_{\vt_i,l},\quad l=1,\dots,n(\vt_1),
$$
hence
$$
\sum_{l\in I_{\vt,1}}(p^{\alpha_{\vt,l}}-p^{\alpha_{\vt,l-1}})\cdot(\partial_\vt b_{\vt,l}-1)=p^{\alpha_0}\cdot\sum_{l=1}^{n(\vt_1)}(p^{\alpha_{\vt_1,l}}-p^{\alpha_{\vt_1,l-1}})\cdot(\partial_{\vt_1}b_{\vt_1,l}-1).
$$
Suppose for the sake of induction that for some fixed $i=2,\dots,m$  and each $j=1,\dots,i-1$ we have $n(\vt_j)=|I_{\vt,j}|$ and
\begin{equation}\label{eq: ind}
\sum_{l\in I_{\vt,j}}(p^{\alpha_{\vt,l}}-p^{\alpha_{\vt,l-1}})\cdot(\partial_\vt b_{\vt,l}-1)=p^{\alpha_{j-1}}\cdot\sum_{l=1}^{n(\vt_j)}(p^{\alpha_{\vt_j,l}}-p^{\alpha_{\vt_j,l-1}})\cdot(\partial_{\vt_j}b_{\vt_j,l}-1).
\end{equation}
We consider morphism $f_{\vt_n}\circ\dots\circ f_{\vt_i}:A_{\vt_i}\to f_{\vt_n}(A_{\vt_n})$ and study the first $|I_{\vt,i}|$ break-points of its profile at point $\eta^{T_{\vt_i}}_{\rho^{p^{\alpha_{i-1}}}}$ (which is the image of $\eta^{T_{\vt_1}}_\rho$ by a map $f_{\vt_{i-1}}\circ\dots\circ f_{\vt_1}$). On one side, by applying the same reasoning as before, we conclude that $n(\vt_i)=|I_{\vt,i}|$ and that the break-points are given by 
\begin{equation}\label{eq: br 1}
b_{\vt_i,l}(\rho^{p^{\alpha_{i-1}}}),\quad l=1,\dots,n(\vt_i).
\end{equation}
On the other side, by using that 
$$
\pf_{f_{\vt_n}\circ\dots\circ f_{\vt_i}}(\eta^{T_{\vt_i}}_{\rho^{p^{\alpha_{i-1}}}})=\pf_{f_\vt}(\eta^{T_\vt}_\rho)\circ\left(\pf_{f_{\vt_{i-1}}\circ\dots\circ f_{\vt_1}}(\eta^{T_{\vt}}_\rho)\right)^{-1},
$$
and the first part of inductive hypothesis \eqref{eq: ind} we obtain that the break-points are given by  
\begin{equation}\label{eq: br 2}
\big(b_{\vt,l}(\rho)\big)^{p^{\alpha_{i-1}}},\quad l\in I_{\vt,i},
\end{equation}
and moreover, the local degrees of the profile of $f_{\vt_n}\circ\dots\circ f_{\vt_i}$ at $\eta^{T_{\vt_i}}_{\rho^{p^{\alpha_{i-1}}}}$ are given by numbers $p^{\alpha_{\vt,l}-\alpha_{i-1}}$, for $l\in I_{\vt,i}$, as $p^{\alpha_{i-1}}$ is the degree of $\pf_{f_{\vt_{i-1}}\circ\dots\circ f_{\vt_1}}$ (because $f_{i-1}\circ\dots\circ f_1$ is purely inseparable at $y_1$, its degree $p^{\alpha_{i-1}}$ is also the degree of $f_{\vt_{i-1}}\circ\dots\circ f_{\vt_1}$, see Remark \ref{rmk: insep degree} \emph{(1)}).
By comparing \eqref{eq: br 1} and \eqref{eq: br 2} we conclude that 
$$
\sum_{l\in I_{\vt,i}}(p^{\alpha_{\vt,l}-\alpha_{i-1}}-p^{\alpha_{\vt,l-1}-\alpha_{i-1}})\cdot(\partial_\vt b_{\vt,i}-1)=\sum_{l=1}^{n(\vt_i)}(p^{\alpha_{\vt_i,l}}-p^{\alpha_{\vt_i,l-1}})\cdot(\partial_{\vt_i}b_{\vt_i,l}-1),
$$
which by multiplying by $p^{\alpha_{i-1}}$ becomes
\begin{equation}\label{eq: br eq}
 \sum_{l\in I_{\vt,i}}(p^{\alpha_{\vt,l}}-p^{\alpha_{\vt,l-1}})\cdot(\partial_\vt b_{\vt,i}-1)=p^{\alpha_{i-1}}\cdot\sum_{l=1}^{n(\vt_i)}(p^{\alpha_{\vt_i,l}}-p^{\alpha_{\vt_i,l-1}})\cdot(\partial_{\vt_i}b_{\vt_i,l}-1).
\end{equation}
By induction, the previous equation holds for each $i=1,\dots,m$.

We continue the proof of the theorem. If $i\neq 0$, then we note that Lemma \ref{lem: harmon simply ram} applied to the morphism $f_i:U_i\to U_{i+1}$ together with equation \eqref{eq: br eq} yields 
$$
\big(p^{\alpha_{i}-\alpha_{i-1}}-1\big)\cdot\big(2-2\cdot g(y_i)\big)=\sum_{\vt\in T_yY}\Big(\frac{1}{p^{\alpha_{i-1}}}\cdot\sum_{j\in I_{\vt,i}}(p^{\alpha_j}-p^{\alpha_{j-1}})\cdot\big(\partial_{\vt}b_{\vt,j}-1\big)\Big),
$$
 and since $g(y_i)=g(y)$ (Remark \ref{rmk: insep degree} \emph{(2)}) the theorem follows in this case by multiplying both sides by $p^{\alpha_{i-1}}$.
 
 If $i=0$, then Lemma \ref{lem: harmon sep} applied to the morphism $f_{n+1}:U_{n+1}\to f(U)$ together with equation \eqref{eq: br eq} and the fact $\deg(f_{n+1})=\fs_{f,y}$ implies
 $$
 \fs_{f,y}\cdot\big(2-2\cdot g(x)\big)-\big(2-2\cdot g(y_{n+1})\big)=\sum_{\vt\in T_yY}\Big(\frac{1}{p^{\alpha_{n}}}\cdot\sum_{j\in I_{\vt,0}}(p^{\alpha_j}-p^{\alpha_{j-1}})\cdot\big(\partial_{\vt}b_{\vt,j}-1\big)+d_{\vt}-p^{\alpha_{\vt,n(\vt)}}\Big).
 $$
 
The theorem follows by multiplying the previous equation by $p^{\alpha_n}$ and by noticing that $\fs_{f,y}\cdot p^{\alpha_n}=\deg(f,y)$ and once again that $g(y_{n+1})=g(y)$ by Remark \ref{rmk: insep degree} \emph{(2)}.
\end{proof}

\begin{rmk}
 We note that by summing the equations in Theorem \ref{thm: loc harmonicity} for $i=0,\dots,n$, and using Lemma \ref{lem: partial exists 1} we obtain the Riemann-Hurwitz formula in Theorem \ref{thm: RH}.
\end{rmk}

\bibliographystyle{plain}
\bibliography{biblio}

\begin{thebibliography}{10}

\bibitem{Ba-Bo}
Francesco Baldassari and Velibor Bojkovi\'c.
\newblock Metric uniformization of morphisms of {B}erkovich curves via $p$-adic
  differential equations.
\newblock {\em Arxiv preprint: https://arxiv.org/abs/1901.07644}.

\bibitem{Ber90}
Vladimir~G. Berkovich.
\newblock {\em Spectral Theory and Analytic Geometry Over Non-Archimedean
  Fields}.
\newblock American Mathematical Society, 1990.

\bibitem{BerCoh}
Vladimir~G Berkovich.
\newblock {\'E}tale cohomology for non-archimedean analytic spaces.
\newblock {\em Publications Math{\'e}matiques de l'IH{\'E}S}, 78(1):5--161,
  1993.

\bibitem{BojRH}
Velibor Bojkovi\'c.
\newblock Riemann-{H}urwitz formula for finite morphisms of $p$-adic curves.
\newblock {\em Mathematische Zeitschrift}, 288:1165--1193, 2018.

\bibitem{BojPoi}
Velibor Bojkovi\'c and J\'er\^ome Poineau.
\newblock Pushforward formula for $p$-adic differential equations.
\newblock {\em To appear in American Journal of Mathematics, preprint: Arxiv:
  https://arxiv.org/abs/1703.04188}.

\bibitem{CTT14}
Adina Cohen, Michael Temkin, and Dmitri Trushin.
\newblock Morphisms of {B}erkovich curves and the different function.
\newblock {\em Advances in Mathematics}, 303:800--858, 2016.

\bibitem{Col03}
Robert~F. Coleman.
\newblock Stable maps of curves.
\newblock {\em Documenta Mathematica. Extra Volume Kato}, pages 217--225, 2003.

\bibitem{Duc-book}
Antoine Ducros.
\newblock La structure des courbes analytiques.
\newblock {\em Manuscript available at www. math. jussieu. fr/~ducros}.
\newblock Accessed on October, 2019.

\bibitem{Har77}
Robin Hartshorne.
\newblock {\em Algebraic geometry}, volume~52 of {\em Graduate Texts in
  Mathematics}.
\newblock Springer-Verlag, New York-Heidelberg, 1977.

\bibitem{SerreLF}
Jean-Pierre Serre.
\newblock {\em Local fields}, volume~67 of {\em Graduate Texts in Mathematics}.
\newblock Springer-Verlag, New York-Berlin, 1979.
\newblock Translated from the French by Marvin Jay Greenberg.

\bibitem{TeMu}
Michael Temkin.
\newblock Metric uniformization of morphisms of {B}erkovich curves.
\newblock {\em Advances in Mathematics}, 317:438--472, 2017.

\end{thebibliography}

\end{document}